\documentclass[11pt]{article}

\usepackage{epsfig}
\usepackage{subfigure}
\usepackage{graphicx}
\usepackage{latexsym}
\usepackage{graphicx}
\usepackage{latexsym, bm, amsthm, amsmath}
\usepackage{algorithm}
\usepackage{algorithmic}
\usepackage{xcolor}
\usepackage{multirow,booktabs}
\usepackage{contour,soul,color}
\usepackage[colorlinks=true]{hyperref}
\usepackage{booktabs,cases}
\usepackage{tabularx,multirow} % Table
\usepackage{cite}
\usepackage{amsfonts}
\usepackage{epstopdf}%%用于带有图的文件可以直接由PDFTeXify编译

\hypersetup{urlcolor=blue,citecolor=blue,linkcolor=blue}

\usepackage{amsmath}
\numberwithin{equation}{section}

\newtheorem{lemma}{Lemma}[section]

\newtheorem{theorem}{Theorem}[section]

\newsavebox{\tablebox}

\begin{document}
\title{{\Huge Second-order linear structure-preserving modified finite volume schemes for the regularized long-wave equation}}
\author{Qi Hong$^{a}$, ~Jialing Wang$^{b}$, ~Yuezheng Gong$^{c,*}$ \\ \\
$^{a}$ Graduate School of China Academy of Engineering Physics,\\
, Beijing 100088, China\\
$^{b}$ School of Mathematics and Statistics, \\
Nanjing University of Information Science and Technology,\\
Nanjing 210044 , China\\
$^{c}$ College of Science, Nanjing University of Aeronautics and Astronautics\\
Nanjing 210016, China\\}
\date{}
\maketitle

\begin{abstract}
In this paper,  based on the weak form of the Hamiltonian formulation of the regularized long-wave equation  and a novel approach of transforming the original Hamiltonian energy into a quadratic functional, a fully implicit and three linear-implicit energy conservation numerical schemes are respectively proposed. The resulting numerical schemes are proved theoretically to satisfy the energy conservation law in the discrete level. Moreover, these linear-implicit schemes are efficient in practical computation because only a linear system need to be solved at each time step.
The proposed schemes are both second order accurate in time and space.
Numerical experiments are presented to show  all the proposed schemes have satisfactory performance in providing accurate solution and the remarkable energy-preserving property.
\end{abstract}

{\bf Keywords:} modified finite volume method, discrete variational derivative method, linear scheme, regularized long-wave equation, conservation laws, quadratic invariant.
             %\vskip.5cm
%\noindent\rule[2mm]{\textwidth}{.3pt}
%%%%%%%%%%%%%%%%%%%%%%%%%%%%%%%%%%%%%%%%%%%%%%%%%%%%%%%%%%%%%%%%%%%%%%%%%%%%%%%%%%%
%%%%%%%%%%%%%%%%%%%%%%%%%%%%%%%%%%%%%%%%%%%%%%%%%%%%%%%%%%%%%%%%%%%%%%%%%%%%%%%%%%%
\begin{figure}[b]
\small \baselineskip=10pt
\rule[2mm]{1.8cm}{0.2mm} \par
$^{*}$Corresponding author.\\
E-mail address: gongyuezheng@nuaa.edu.cn~(Yuezheng Gong).
\end{figure}

\pagestyle{myheadings}
\markboth{\hfil %name
   \hfil \hbox{}}
{\hbox{} \hfil %title
Qi Hong, Jialing Wang and Yuezheng Gong \hfil}
%%%%%%%%%%%%%%%%%%%%%%%%%%%%%%%%%%%%%%%%%%%%%%%%%%%%%%%%%%%%%%%%%%%%%%%%%%%%%%%%%%%%%%%%%%%

\section{Introduction}\label{sec;introduction}
The regularized long-wave (RLW) type equation
\begin{align}\label{model-eq}
u_t+au_x-\sigma u_{xxt}+(F^{'}(u))_x=0,\quad F(u)=\dfrac{\gamma}{6}u^3,
\end{align}
where $a$, $\sigma$ and $\gamma$ are positive constants.
The RLW equation was proposed first by Peregrine
\cite{Peregrine1966}. Benjamin, Bona, and Mahony derived it as alternative to the Korteweg-de Vries equation for describing unidirectional propagation of weakly along dispersive wave \cite{Benjamin1972}.
The RLW is very important in physics media since it describes phenomena with weak nonlinearity and dispersion
waves, including nonlinear transverse waves in shallow water, ion-acoustic and magneto hydrodynamic waves in plasma and phonon packets in nonlinear crystals.
It admits three conservation laws \cite{Olver1979} given by
\begin{align*}
\mathcal{I}_1=\int_{-\infty}^{\infty}udx,\quad
\mathcal{I}_2=\int_{-\infty}^{\infty}\left(u^2+\sigma u_x^2\right)dx,\quad
\mathcal{I}_3=\int_{-\infty}^{\infty}\left(\dfrac{\gamma}{6}u^3+\dfrac{a}{2}u^2\right)dx,
\end{align*}
which correspond to mass, momentum and energy, respectively. As the RLW equation and its variants have been solved analytically for a restricted set of boundary and initial conditions, its numerical solution has been the subject of many literatures
\cite{Dag2004,Dehghan2011,Dogan2002,Gao2015,Gu2008,Guo1988,Lu2018,Luo1999,Mei2015,MeiChen2012,Meigaocheng2015,ZakiSI2001}.

However, some numerical algorithms obtained using standard approaches could not inherit certain
invariant quantities in the continuous dynamical systems. A numerical method which can
preserve at least some of structural properties of systems is called geometric integrator or structure
preserving algorithm. Nowadays, it has been a criterion to judge the success of the numerical simulation.
When discretizing such a conservative system in space and time, it is a natural
idea to design numerical schemes that preserve rigorously a discrete invariant that is an
equivalence of the continuous one, since they often yield physically correct results and
numerical stability \cite{Bubb2003Book}.

To our knowledge, Sun and Qin \cite{Sun2004} constructed a multi-symplectic Preissman scheme by using the implicit midpoint rule both in space and time. Cai \cite{Cai2009} developed a 6-point multi-symplectic Preissman scheme. An explicit 10-point multi-symplectic Euler-box scheme for the RLW equation was proposed in \cite{02Cai2009},  and the author applied this scheme to solve the modified
RLW in \cite{Cai2010}.
Based on the multi-symplectic Euler box scheme and Preissman box scheme, Li and Sun \cite{Li2013} proposed a new multi-symplectic Euler box scheme for RLW equation with the accuracy of $\mathcal{O}(h^2+\tau^2)$.
In some fields, sometimes it is more convenient to construct numerical algorithms that preserve the energy conservation law rather than the sympletic or multi-symplectic one \cite{Hairer2006Book}.
Fortunately, the energy-preserving methods have been developed on numerical partial differential equations (PDEs) (see \cite{Eidnes2017,Gong2014,Matsuo2007,Quispel2008} for examples).
Furhata \cite{Furihata1999} presented the discrete variational derivative methods (DVDM) for a large class of PDEs that inherit energy conservation or dissipation property. Matsuo and Furihata \cite{Furihata2001} generalized the DVDM for complex-valued nonlinear PDEs.
However, most existing energy-preserving methods of the RLW equation are fully implicit \cite{CaiHong2017,Hammad2015}, and so on. Thus, it needs to use an iteration technique to evaluate numerical solution. Recently, Dahlby and Owren proposed a general framework
for deriving linear-implicit integral-preserving numerical methods for PDEs with polynomial invariants \cite{Dahlby2011}. Inpsired by the contribution in \cite{Dahlby2011}, we will develop linear energy conservation schemes for the RLW.

In this article, we design four energy-preserving numerical methods based on the variational technique (such as finite volume method) for the space semi-discretization. Note that this approach holds the conservation of a semi-discrete energy which is a discrete approximation of the original one in the continuous conservative system . It is a quite natural idea to gain a full discrete scheme that can preserve the original energy as accurately as possible.
To this end, we will use the DVDM in time, which can preserve the semi-discrete energy exactly. First, we apply one-point DVDM in temporal direction. But the resulting numerical scheme is a fully implicit and an
iterative technique is needed in numerical computations. In order to overcome the disadvantage, a novel linear energy conservation scheme is proposed by the DVDM with two points numerical solutions.

In addition, we propose a new idea of transforming the original Hamiltonian energy into a quadratic functional to derive the energy stable schemes. Based on this strategy, we still discrete the weak form of the equivalent formulation for the RLW equation in space by the so-called modified finite volume method (mFVM), which is obtained by the linear finite volume method while the second-order derivative is approximated by the central finite difference. Then, we consider the linear-implicit Crank-Nicolson and Leap-frog scheme  to discretize the conservative semi-discrete system in temporal direction and two second-order linear energy-preserving schemes are obtained readily. In our proposed linear-implicit methods, we just need to
solve one linear system at each time step, which is less inexpensive than any nonlinear
schemes. Compared with several existing works, numerical tests have confirmed this conclusion. Finally, we strictly prove that the four schemes mentioned above preserve the discrete energy exactly. Numerical results demonstrate the advantages of the methods are presented.

The paper is organized as follows. In Section 2, some elementary notations and definitions are briefly introduced firstly and then we describe in detail the idea of the mFVM discretization in space. Section 3 is devoted to designing energy-preserving discretization schemes, which can preserve the discrete energy exactly. The conservation of the quadratic invariant is developed in section 4. Numerical examples are presented to show the validity of theoretical results in Section 5. Finally, we give a concluding remark in Section 6.

\section{Structure-preserving Spatial Discretization based on mFVM }
In this section, we construct a semi-discrete scheme by mFVM for the RLW equation in a domain $\Omega=[a,b]$. The semi-discrete scheme preserves the corresponding energy conservation law at the semi-discrete level.

The equation \eqref{model-eq} can be rewritten the following formulation
\begin{align}\label{RLW-Hamilton-eq1}
&(1-\sigma\partial^2_x)u_t=-\partial_x\dfrac{\delta \mathcal{H}}{\delta u},
\end{align}
where $\frac{\delta \mathcal{H}}{\delta u}$ is a variational derivative of the Hamiltonian function
\begin{align*}
\mathcal{H}(u,u_x)=\int H(u,u_x)dx,\quad H(u,u_x)=\frac{\gamma}{6}u^3+\frac{a}{2}u^2.
\end{align*}
Let $H_p^1(\Omega)=\left\{u\in H^1(\Omega):u(x)=u(x+b-a)\right\}$. Further, $\forall\;u,v\in L^2(\Omega)$, $(u,v)=\int_{\Omega}uvdx$.
A weak formulation of the Galerkin discretization for \eqref{RLW-Hamilton-eq1} stats from finding $u\in H_p^1(\Omega)$, such that
\begin{align}\label{weak-form-eq}
\Big((1-\sigma\partial_x^2)u_t,v\Big)=-\Big(\partial_x\dfrac{\delta\mathcal{H}}{\delta u},v\Big),\quad \forall\; v\in H_p^{1}(\Omega).
\end{align}
Obviously, the $\mathcal{H}$-conservation law can be explicitly obtained by formal calculation
\begin{align*}
\dfrac{d}{dt}\mathcal H=\dfrac{d}{dt}\int Hdx=(\dfrac{\partial H}{\partial u},u_t)
=(\dfrac{\delta \mathcal{H}}{\delta u},-(1-\sigma\partial^2_x)^{-1}\partial_x\dfrac{\delta \mathcal{H}}{\delta u})=0,
\end{align*}
where the second equality is just the chain rule and the skew-adjoint operator of $(1-\sigma\partial^2_x)^{-1}\partial_x$ has been used.

Suppose that $\Omega$ is partitioned into a number of non-overlapped cells that form the so-called primary cell $\Omega_h=\left\{\Omega_i|\Omega_i=[x_i,x_{i+1}],\; i=0, 1, \cdots, N \right\}$ with the mesh size $h=x_{i+1}-x_{i}=(b-a)/N$.  Denote by $x_{i+1/2}$ cell center of cell $[x_i,x_{i+1}]$. In this paper, the model is imposed with a periodic boundary condition. Accordingly, we define a dual grid $\Omega_h^*$ as
\begin{align*}
\Omega_h^*=\{\Omega_0^*\cup\Omega_i^*:i=1,2,\cdots N-1\},
\end{align*}
where $\Omega_0^*=[x_0,x_{1/2}]\cup [x_{N-1/2},x_N]$ and $\Omega_i^*=[x_{i-1/2},x_{i+1/2}]$.
The trial function space $U_h$ is taken as the linear element space and the corresponding basis functions are given by
\begin{align*}
\psi_0=
\begin{cases}
(x_1-x)/h,\quad &x_0\leq x\leq x_1,\\
(x-x_{N-1})/h,\quad &x_{N-1}\leq x\leq x_N,\\
0,&\mathrm{otherwise},
\end{cases}
\quad
\psi_i(x)=
\begin{cases}
(x-x_{i-1})/h,\quad &x_{i-1}\leq x\leq x_{i},\\
(x_{i+1}-x)/h,\quad &x_{i}\leq x\leq x_{i+1},\\
0,&\mathrm{otherwise},
\end{cases}
\end{align*}
where $i=1,2,\cdots,N-1$.
For $\forall\; u_h\in U_h$, we have
\begin{align*}
u_h=\sum_{i=0}^{N-1}u_i\psi_i(x),
\end{align*}
where $u_i=u_h(x_i,t)$.
Define the test function space $V_h$ as
\begin{align*}
V_h=\mathrm{Span}\left\{\chi_i(x)|i=0,1,\cdots,N-1\right\},
\end{align*}
where $\chi_i$ denotes the characteristic function of dual cell $\Omega_i^*$ associated with $x_i$. We have
\begin{align*}
v_h=\sum_{i=0}^{N-1}v_i\chi_i,\quad \forall\; v_h\in V_h,
\end{align*}
where $v_i=v_h(x_i)$. We also define an interpolation operator $I_h:L^2(\Omega) \rightarrow U_h$ by
$$I_hu=\sum_{i=0}^{N-1}u(x_i,t)\psi_i(x),\quad\forall\;u\in L^2(\Omega).$$
Throughout this paper, the hollow letters $\mathbb{A}, \mathbb{B}, \mathbb{C}, \cdots$
will be used to denote rectangular matrices with a number of columns greater than one, while the bold ones $\mathbf{u}, \mathbf{v}, \mathbf{ w}, \cdots$ represent vectors.
For a Galerkin approximation of \eqref{weak-form-eq}, we look for $u_h(x,t)\in U_h$ such that, for any $v_h\in V_h$,
\begin{align}\label{semi-discrete-eq}
\Big((1-\sigma\partial_x^2)(u_h)_t,v_h\Big)=-\left(\partial_x\dfrac{\delta \mathcal{H}}{\delta u},v_h\right),
\end{align}
where
$\delta\mathcal{H}/\delta u=\gamma I_hu^2/2+au_h$.
Let $X_h=\{\mathbf{u}:\mathbf{u}=(u_0,u_1,\cdots,u_{N-1})^T\}$ be the space of grid function on $\Omega_h$.
Based on the above results and by some simple calculations, we get
\begin{align*}
\int_{\Omega}(u_h)_t\chi_idx&=\int_{x_{i}-\frac{1}{2}}^{x_i}u_{i-1}\psi_{i-1}dx+\int_{x_{i-\frac{1}{2}}}^{x_{i+\frac{1}{2}}}u_i\psi_idx+\int_{x_i}^{x_{i}+\frac{1}{2}}u_{i+1}\psi_{i+1}dx\\
&=\dfrac{h}{8}(\partial_t u_{i-1}+6\partial_t u_{i}+\partial_t u_{i+1}),\\
\int_{\Omega}\partial_x^2(u_h)_t\chi_idx&\approx\int_{\Omega}\delta_x^2(\partial_tu_i)\chi_idx=
\dfrac{\partial_t u_{i+1}-2\partial_t u_i+\partial_t u_{i-1}}{h},\\
\int_{\Omega}\partial_x(\dfrac{\delta \mathcal{H}}{\delta u})\chi_idx&=(\dfrac{\delta\mathcal{H}}{\delta u})_{i+\frac{1}{2}}-(\dfrac{\delta\mathcal{H}}{\delta u})_{i-\frac{1}{2}}=\dfrac{1}{2}\big[(\dfrac{\delta\mathcal{H}}{\delta u})_{i+1}-(\dfrac{\delta\mathcal{H}}{\delta u})_{i-1}\big],
\end{align*}
where the operator $\delta_x^2$ denotes the two-order central difference.
The corresponding discrete inner product and discrete $L^2$ norm are given by
\begin{align*}
(\mathbf{u},\mathbf{v})_h=h\sum_{j=0}^{N-1}u_jv_j,\quad\|\mathbf{u}\|_h=(\mathbf{u},\mathbf{u})_h^{1/2},\quad\forall\;\mathbf{u},\mathbf{v}\in X_h.
\end{align*}
Finally, we define another operator ``$\odot$'' for element by element multiplication between two arrays of same size as
\begin{align*}
(\mathbf{v}\odot\mathbf{w})_j=(\mathbf{w}\odot\mathbf{v})_j=v_jw_j,
\end{align*}
where $\mathbf{v}, \mathbf{w}\in X_h$.
In fact, \eqref{semi-discrete-eq} is equivalent to the following system
\begin{align}\label{semi-discrete-eq1}
(\mathbb{A}-\sigma\mathbb{B})\dfrac{d\mathbf{u}}{dt}=-\mathbb{C}\dfrac{\delta \mathcal{H}_h}{\delta \mathbf{u}},
\end{align}
where $\mathbf{u}=(u_0(t),u_1(t),\cdots,u_{N-1}(t))^T$, $\mathcal{H}_h=h\sum\limits_{j=0}^{N-1}\big(\dfrac{\gamma u_j^3}{6}+\dfrac{au_j^2}{2}\big)$ and
\begin{align*}
\mathbb{A}=\dfrac{h}{8}\left(
\begin{array}{cccccc}
6& 1& 0&  \cdots& 1\\
1& 6& 1&  \cdots& 0\\
\vdots& \vdots& \vdots& \vdots& \vdots \\
1& 0& \cdots& 1& 6
\end{array}
\right),\quad
\mathbb{B}=\dfrac{1}{h}\left(
\begin{array}{ccccc}
-2& 1& 0&  \cdots& 1\\
1& -2& 1&  \cdots& 0\\
\vdots& \vdots&  \vdots& \vdots& \vdots \\
1& 0&  \cdots&\ 1& -2
\end{array}
\right),\quad
\mathbb{C}=\dfrac{1}{2}\left(
\begin{array}{ccccc}
0& 1&  0& \cdots& -1\\
-1& 0& 1&  \cdots& 0\\
\vdots& \vdots&  \vdots& \vdots& \vdots \\
1&\ 0& \cdots& -1& 0
\end{array}
\right).
\end{align*}
\begin{lemma}\label{integ-by-parts-lem}\cite{Gong2017}
	For any real square $\mathbb{A}$ and $\mathbf{u},\mathbf{v}\in X_h$, we have
	\begin{align*}
	(\mathbb{A}\mathbf{u},\mathbf{v})_h=(\mathbf{u},\mathbb{A}^T\mathbf{v})_h.
	\end{align*}
\end{lemma}

\begin{theorem}
	The semi-discrete scheme \eqref{semi-discrete-eq1} conserves the discrete energy conservation law
	\begin{align*}
	\dfrac{d}{dt}\mathcal{H}_h=0,
	\end{align*}
	where $\mathcal{H}_h=h\sum\limits_{j=0}^{N-1}\big(\dfrac{\gamma u_j^3}{6}+\dfrac{au_j^2}{2}\big)$.
\end{theorem}
\begin{proof}
	Computing the discrete inner product on the both sides of \eqref{semi-discrete-eq1} with $\delta \mathcal{H}_h/\delta \mathbf{u}$, we have
	\begin{align}\label{semi-disc-theom-eq1}
	\left(\dfrac{d\mathbf{u}}{dt},\dfrac{\delta \mathcal{H}_h}{\delta \mathbf{u}}\right)_h=-\left((\mathbb{A}-\sigma \mathbb{B})^{-1}\mathbb{C}\dfrac{\delta \mathcal{H}_h}{\delta \mathbf{u}},\dfrac{\delta \mathcal{H}_h}{\delta \mathbf{u}}\right)_h.
	\end{align}
	Using the antisymmetry of matrix $(\mathbb{A}-\sigma \mathbb{B})^{-1}\mathbb{C}$ and Lemma \ref{integ-by-parts-lem}, the right-hand of \eqref{semi-disc-theom-eq1} becomes
	\begin{align}\label{semi-disc-theom-eq2}
	\left((\mathbb{A}-\sigma \mathbb{B})^{-1}\mathbb{C}\dfrac{\delta \mathcal{H}_h}{\delta \mathbf{u}},\dfrac{\delta \mathcal{H}_h}{\delta \mathbf{u}}\right)_h=0.
	\end{align}
	Noth that
	\begin{align}\label{semi-disc-theom-eq3}
	\left(\dfrac{d\mathbf{u}}{dt},\dfrac{\delta \mathcal{H}_h}{\delta \mathbf{u}}\right)_h=\dfrac{d}{dt}\mathcal{H}_h.
	\end{align}
	Combining the \eqref{semi-disc-theom-eq1}, \eqref{semi-disc-theom-eq2} with \eqref{semi-disc-theom-eq3} yields the desired result and completes the proof.
\end{proof}

\section{Energy conservation scheme for the RLW equation}
In this section, we construct two energy-preserving schemes for the RLW equation by the DVDM. One is a fully-implicity energy-preserving (FIEP) scheme, the other is a linear-implicity energy-preserving (LIEP) scheme.

For a positive integer $N_t$, we denote time step $\tau=T/N_t$, $t_n=n\tau,\ 0\leq n\leq N_t$. We define
\begin{align*}
&\delta_t^+{\mathbf u}^n=\dfrac{\mathbf u^{n+1}-\mathbf u^n}{\tau},\quad \delta_t\mathbf{u}^n=\dfrac{\mathbf{u}^{n+1}-\mathbf{u}^{n-1}}{2\tau},\\[0.3cm]
&\mathbf{\overline{u}}^{n+\frac{1}{2}}=\dfrac{3\mathbf{u}^{n}-\mathbf{u}^{n-1}}{2},\quad
\mathbf{u}^{n+\frac{1}{2}}=\dfrac{\mathbf{u}^{n+1}+\mathbf{u}^n}{2},\quad A_t\mathbf{u}^n=\dfrac{\mathbf u^{n+1}+\mathbf{u}^{n-1}}{2}.
\end{align*}

\subsection{Fully-implicit energy-preserving (FIEP) scheme for the RLW equation}
We define a discrete form of the energy and its partial derivative as
\begin{align*}
H_h^n(u_j^n)=\dfrac{\gamma}{6}(u_j^n)^3+\dfrac{a}{2}(u_j^n)^2,\quad
{\mathcal{H}}_h^n=h\sum_{j=0}^{N-1}H_h^n(u^n_j),
\end{align*}
and
\begin{align*}
\dfrac{\partial H_h^n}{\partial(u_j^n,u_j^{n+1})}&=\dfrac{\gamma}{6}\dfrac{(u_j^{n+1})^3-(u_j^n)^3}{u_j^{n+1}-u_j^n}
+\dfrac{a}{2}\dfrac{(u_j^{n+1})^2-(u_j^n)^2}{u_j^{n+1}-u_j^n}\\[0.3cm]
&=\gamma \dfrac{(u_j^{n+1})^2+u_j^{n+1}u_j^n+(u_j^n)^2}{6}+
a\dfrac{u_j^{n+1}+u_j^n}{2}.
\end{align*}
We approximate $\delta\mathcal{H}/\delta u$ by the discrete version of the variational
derivative, i.e.,
\begin{align*}
\dfrac{\delta{\mathcal{H}}_h^n}{\delta(u_j^n,u_j^{n+1})}=\dfrac{\partial H^n_h}{\partial(u_j^n,u_j^{n+1})}-\dfrac{\partial}{\partial x}\left(\dfrac{\partial H_h^n}{\partial((u_x)_j^n,(u_x)_j^{n+1})}\right)=\dfrac{\partial H^n_h}{\partial(u_j^n,u_j^{n+1})}.
\end{align*}
%where $\partial H_h^n/\partial((u_x)_j^{n+1},(u_x)_j^n)=0$.
We discretize the above semi-discrete system \eqref{semi-discrete-eq} in time by DVDM, thus the fully discrete scheme is given by
\begin{align*}%\label{FVM-disrete-eq1}
\Big((1-\sigma \partial_x^2)\delta_t^{+}u^n_h,v_h\Big)=-\left(\partial_x\left(\dfrac{\delta {\mathcal{H}}_h^n}{\delta(u^{n+1},u^n)}\right),v_h\right).
\end{align*}
For simplicity, it can be written as the following compact representation
\begin{align}\label{matrix-form-eq1}
\delta_t^+\mathbf{u}^n=-(\mathbb{A}-\sigma \mathbb{B})^{-1}\mathbb{C}F(\mathbf{u}^n,\mathbf{u}^{n+1}),
\end{align}
where
\begin{align*}
F(\mathbf{u}^n,\mathbf{u}^{n+1})=\dfrac{\gamma}{6}(\mathbf{u}^{n+1}\odot\mathbf{u}^{n+1}
+\mathbf{u}^{n}\odot\mathbf{u}^{n+1}+\mathbf{u}^n\odot\mathbf{u}^n)+\dfrac{a}{2}(\mathbf{u}^{n+1}+\mathbf{u}^n).
\end{align*}
\begin{theorem}\label{theorem-FIEP}
	Under the periodic boundary condition, the FIEP scheme preserves the discrete
	energy conservation law
	\begin{align*}
	\mathcal{H}_h^{n+1}=\mathcal{H}_h^n,\quad \forall\; n\geq 0,
	\end{align*}
	where $\mathcal{H}_h^n=h\sum\limits_{j=0}^{N-1}\big(\dfrac{\gamma (u^n_j)^3}{6}+\dfrac{a(u^n_j)^2}{2}\big)$ is called the total energy in the discrete level.
\end{theorem}
\begin{proof}
	Noticing that $(\mathbb{A}-\sigma\mathbb{B})^{-1}\mathbb{C}$ is a skew-symmetric matrix and using Lemma \ref{integ-by-parts-lem}, we have
	\begin{align*}
	\left((\mathbb{A}-\sigma\mathbb{B})^{-1}\mathbb{C}F(\mathbf{u}^n,\mathbf{u}^{n+1}),F(\mathbf{u}^n,\mathbf{u}^{n+1})\right)_h=0.
	\end{align*}
	Therefore, computing the discrete inner product on the both side of \eqref{matrix-form-eq1} with $F(\mathbf{u}^n,\mathbf{u}^{n+1})$, we have
	\begin{align*}
	0=(\delta_t^{+}\mathbf{u}^n,F(\mathbf{u}^n,\mathbf{u}^{n+1}))_h=\dfrac{1}{\tau}(\mathcal{H}_h^{n+1}-\mathcal{H}_h^{n}).
	\end{align*}
	The proof is complete.
\end{proof}

\subsection{Linear-implicity energy-preserving (LIEP) scheme for the RLW equation}
Different from the previous one, we define a new discrete energy and its partial derivative as follows:
\begin{align*}
{\mathcal{H}}_h^n=h\sum_{j=0}^{n-1} H_h^n(u_j^n,u_j^{n+1}),\quad
H_h^n(u_j^n,u_j^{n+1})=\gamma\dfrac{u_j^{n+1}u_j^n(u_j^{n+1}+u_j^n)}{12}+a\dfrac{u_j^nu_j^{n+1}}{2},
\end{align*}
and
\begin{align*}
\dfrac{\partial H_h^n}{\partial(u_j^{n-1},u_j^n,u_j^{n+1})}=\dfrac{ H_h^n(u_j^{n+1},u_j^n)- H_h^{n-1}(u_j^n,u_j^{n-1})}{\frac{1}{2}(u_j^{n+1}-u_j^{n-1})}.
\end{align*}
Therefore, the approximation of $\delta\mathcal{H}/\delta u$ can be given by
\begin{align*}
\dfrac{\delta\mathcal{H}^n_h}{\delta(u_j^{n-1},u_j^n,u_j^{n+1})}=\dfrac{\partial H_h^n}{\partial(u_j^{n-1},u_j^n,u_j^{n+1})}=au_j^n+\gamma\dfrac{u_j^n(u_j^{n-1}+u_j^n+u_j^{n+1})}{6}.
\end{align*}
Applying the novel DVDM mentioned above in the temporal direction of the semi-discrete form \eqref{semi-discrete-eq}, we can obtain the full discrete scheme
\begin{align}\label{scheme2-eq1}
\Big((1-\sigma \partial_x^2)\delta_tu^n_h,v_h\Big)=-\left(\partial_x\left(\dfrac{\delta\mathcal{H}_h^n}{\delta(u^{n-1},u^n,u^{n+1})}\right),v_h\right).
\end{align}
Obviously, we can rewrite \eqref{scheme2-eq1} in the following compact form
\begin{align}\label{scheme2-eq2}
(\mathbb{A}-\sigma \mathbb{B})\delta_t\mathbf{u}^n=-\mathbb{C}G(\mathbf{u}^{n-1},\mathbf{u}^n,\mathbf{u}^{n+1}),
\end{align}
where
\begin{align*}
G(\mathbf{u}^{n-1},\mathbf{u}^n,\mathbf{u}^{n+1})=a\mathbf{u}^n+\dfrac{\gamma}{6}\mathbf{u}^n\odot(\mathbf{u}^{n+1}+\mathbf{u}^n+\mathbf{u}^{n-1}).
\end{align*}
Note that the above system is a three-level scheme, where $\mathbf{u}^1$ can be given by a suitable two-level scheme
\begin{align*}%\label{set-up-eq}
(\mathbb{A}-\sigma \mathbb{B})\delta_t^+\mathbf{u}^0=-\mathbb{C}F(\mathbf{u}^0,\mathbf{u}^1),
\end{align*}
where
\begin{align*}
F(\mathbf{u}^0,\mathbf{u}^1)=\dfrac{\gamma}{6}\big(\mathbf{u}^{1}\odot\mathbf{u}^{1}+\mathbf{u}^{1}\odot\mathbf{u}^0+\mathbf{u}^0\odot\mathbf{u}^0\big)+\dfrac{a}{2}(\mathbf{u}^{1}+\mathbf{u}^0).
\end{align*}

\begin{theorem}
	Under the periodic boundary condition, the LIEP scheme preserves the discrete
	energy conservation law
	\begin{align*}
	\mathcal{H}_h^{n+1}=\mathcal{H}^n_h,
	\end{align*}
	where
	\begin{align*}
	\mathcal{H}_h^n=h\sum\limits_{j=0}^{N-1}\Big(\gamma\dfrac{u_j^{n+1}u_j^n(u_j^{n+1}+u_j^n)}{12}+a\dfrac{u_j^nu_j^{n+1}}{2}\Big).
	\end{align*}
\end{theorem}

\begin{proof}
	The proof is similar to the theorem \ref{theorem-FIEP}, thus it is omitted.
\end{proof}

\section{The conservation of the quadratic invariant}
In order to develop an equivalent system with a quadratic energy functional, we introduce a new intermediate variable $v=u^2$ and rewrite the RLW equation \eqref{model-eq} into the following equivalent formulation:
\begin{align*}
&(1-\sigma\partial_x^2)u_t=-\partial_x\left(\dfrac{\gamma}{6}v+au+\dfrac{\gamma}{3}u^2\right),\\[0.3cm]
&(1-\sigma\partial_x^2)v_t=-2u\partial_x\left(\dfrac{\gamma}{6}v+au+\dfrac{\gamma}{3}u^2\right),
\end{align*}
where the corresponding quadratic energy conservation law
\begin{align*}
\mathcal{H}(t)=\int\left(\dfrac{\gamma}{6}uv+\dfrac{a}{2}u^2\right)dx\equiv \mathcal{H}(0).
\end{align*}
The so-called modified finite volume discretization in space of the system is based on the Galerkin formulation, thus we have
\begin{align*}
\begin{array}{ll}
\Big((1-\sigma\partial_x^2)(u_h)_t,w_1\Big)=-\left(\partial_x\big(\dfrac{\gamma}{6}v_h+au_h+\dfrac{\gamma}{3}u_h^2\big),w_1\right),\quad&\forall\;w_1\in V_h,\\[0.3cm]
\Big((1-\sigma\partial_x^2)(v_h)_t,w_2\Big)=-\left(2u_h\partial_x\big(\dfrac{\gamma}{6}v_h+au_h+\dfrac{\gamma}{3}u_h^2\big),w_2\right),\quad&\forall\;w_2\in V_h.
\end{array}
\end{align*}
The semi-discrete system is equivalent to the following ODE system
\begin{align}\label{new-semi-discrete-eq}
\begin{split}
&(\mathbb{A}-\sigma \mathbb{B})\dfrac{d}{dt}\mathbf{u}=-\mathbb{C}\left(\dfrac{\gamma}{6}\mathbf{v}+a\mathbf{u}\right)-\dfrac{\gamma}{3}\mathbb{C}(\mathbf{u}\odot\mathbf{u}),\\[0.3cm]
&(\mathbb{A}-\sigma\mathbb{B})\dfrac{d}{dt}\mathbf{v}=-2\mathbf{u}\odot \mathbb{C}\left(\dfrac{\gamma}{6}\mathbf{v}+a\mathbf{u}\right)-\dfrac{2}{3}\mathbf{u}\odot\mathbb{C}(\mathbf{u}\odot\mathbf{u}).
\end{split}
\end{align}

\subsection{Linear-implicity Crank-Nicolson (LICN) scheme for the RLW equation}
Applying the linear-implicit Crank-Nicolson scheme in temporal direction for the semi-discrete system \eqref{new-semi-discrete-eq}, we have
\begin{align}
&(\mathbb{A}-\sigma\mathbb{B})\delta_t^+\mathbf{u}^n=-\mathbb{C}\left(\dfrac{\gamma}{6}\mathbf{v}^{n+\frac{1}{2}}+a\mathbf{u}^{n+\frac{1}{2}}\right)
-\dfrac{\gamma}{3}\mathbb{C}\mathrm{diag}(\mathbf{\bar{u}}^{n+\frac{1}{2}})\mathbf{u}^{n+\frac{1}{2}},\label{linear-scheme-eq1}\\[0.3cm]
&(\mathbb{A}-\sigma\mathbb{B})\delta_t^+\mathbf{v}^n=-2\mathrm{diag}(\mathbf{\bar{u}}^{n+\frac{1}{2}})\mathbb{C}\left(\dfrac{\gamma}{6}\mathbf{v}^{n+\frac{1}{2}}
+a\mathbf{u}^{n+\frac{1}{2}}\right)-\dfrac{2}{3}\mathrm{diag(\mathbf{\bar u}^{n+\frac{1}{2}})}\mathbb{C}\mathrm{diag}(\mathbf{\bar u}^{n+\frac{1}{2}})\mathbf{u}^{n+\frac{1}{2}},\notag
\end{align}
where $n\geq 1$. Obviously, \eqref{linear-scheme-eq1} is also a three-level scheme which can not start by itself. Thus, the first step $\mathbf{u}^1$ can be computed by a proper two-level scheme as follows
\begin{align}\label{start-linear-scheme-eq}
\begin{split}
&(\mathbb{A}-\sigma\mathbb{B})\delta_t^+\mathbf{u}^0=-\mathbb{C}\left(\dfrac{\gamma}{6}\mathbf{v}^{\frac{1}{2}}+a\mathbf{u}^{\frac{1}{2}}\right)-\dfrac{\gamma}{3}\mathbb{C}\mathrm{diag}(\mathbf{u}^{0})\mathbf{u}^{\frac{1}{2}},\\[0.3cm]
&(\mathbb{A}-\sigma\mathbb{B})\delta_t^+\mathbf{v}^0=-2\mathrm{diag}(\mathbf{u}^{0})\mathbb{C}\left(\dfrac{\gamma}{6}\mathbf{v}^{\frac{1}{2}}+a\mathbf{u}^{\frac{1}{2}}\right)-\dfrac{2}{3}\mathrm{diag(\mathbf{u}^{0})}\mathbb{C}\mathrm{diag}(\mathbf{u}^{0})\mathbf{u}^{\frac{1}{2}},\\[0.3cm]
&\mathbf{v}^0=\mathbf{u}^0\odot\mathbf{u}^0.
\end{split}
\end{align}
\begin{theorem}
	The scheme \eqref{linear-scheme-eq1} and \eqref{start-linear-scheme-eq} satisfy the discrete
	energy conservation law
	\begin{align}\label{energy-conservation-law}
	\mathcal{H}_h^n=\mathcal{H}_h^0,\quad\forall\;n\geq 0,
	\end{align}
	where $\mathcal{H}_h^n=h\sum\limits_{j=0}^{N-1}\big(\dfrac{\gamma u^n_jv^n_j}{6}+\dfrac{a(u^n_j)^2}{2}\big)$.
\end{theorem}
\begin{proof}
	Taking the discrete inner product of the first equation of \eqref{linear-scheme-eq1} with $\frac{\gamma}{6}\mathbf{v}^{n+\frac{1}{2}}+a\mathbf{u}^{n+\frac{1}{2}}$, yields
	\begin{align}\label{quadratic-theorem-eq1}
	\left(\delta_t^+\mathbf{u}^n,\dfrac{\gamma}{6}\mathbf{v}^{n+\frac{1}{2}}+a\mathbf{u}^{n+\frac{1}{2}}\right)_h
	&=-\dfrac{\gamma^2}{18}\left((\mathbb{A}-\sigma\mathbb{B})^{-1}\mathbb{C}\mathrm{diag}(\bar{\mathbf{u}}^{n+\frac{1}{2}})\mathbf{u}^{n+\frac{1}{2}},\mathbf{v}^{n+\frac{1}{2}}\right)_h.
	\end{align}
	because of the antisymmetry of $(\mathbb{A}-\sigma\mathbb{B})^{-1}\mathbb{C}$.
	In the similar way, taking the discrete inner product of the second equation of \eqref{linear-scheme-eq1} with $\frac{\gamma}{6}\mathbf{u}^{n+\frac{1}{2}}$, we have
	\begin{align}\label{inner-produce-eq2}
	\left(\delta_t^+\mathbf{v}^n,\dfrac{\gamma}{6}\mathbf{u}^{n+\frac{1}{2}}\right)_h
	=-\dfrac{\gamma^2}{18}\left(\mathrm{diag}(\bar{\mathbf{u}}^{n+\frac{1}{2}})(\mathbb{A}-\sigma\mathbb{B})^{-1}\mathbb{C}\mathbf{v}^{n+\frac{1}{2}},\mathbf{u}^{n+\frac{1}{2}}\right)_h
	\end{align}
	due to the antisymmetry of $\mathbb{C}$.
	Using Lemma \ref{integ-by-parts-lem}, we have
	\begin{align*}
	\left((\mathbb{A}-\sigma\mathbb{B})^{-1}\mathbb{C}\mathrm{diag}(\bar{\mathbf{u}}^{n+\frac{1}{2}})\mathbf{u}^{n+\frac{1}{2}},\mathbf{v}^{n+\frac{1}{2}}\right)_h
	+\left(\mathrm{diag}(\bar{\mathbf{u}}^{n+\frac{1}{2}})(\mathbb{A}-\sigma\mathbb{B})^{-1}\mathbb{C}\mathbf{v}^{n+\frac{1}{2}},\mathbf{u}^{n+\frac{1}{2}}\right)_h=0.
	\end{align*}
	So, adding \eqref{quadratic-theorem-eq1} and \eqref{inner-produce-eq2}, we obtain
	\begin{align}
	\left(\delta_t^+\mathbf{u}^n,\dfrac{\gamma}{6}\mathbf{v}^{n+\frac{1}{2}}+a\mathbf{u}^{n+\frac{1}{2}}\right)_h+\left(\delta_t^+\mathbf{v}^n,\dfrac{\gamma}{6}\mathbf{u}^{n+\frac{1}{2}}\right)_h=0,
	\end{align}
	i.e.,
	\begin{align*}
	&\dfrac{\gamma}{12\tau}\left((\mathbf{u}^{n+1})^T\mathbf{v}^{n+1}-(\mathbf{u}^{n})^T\mathbf{v}^{n}+(\mathbf{u}^{n+1})^T\mathbf{v}^n-(\mathbf{u}^n)^T\mathbf{v}^{n+1}\right)
	+\dfrac{a}{2\tau}\left((\mathbf{u}^{n+1})^T\mathbf{u}^{n+1}-(\mathbf{u}^n)^T\mathbf{u}^n\right)\\
	&+\dfrac{\gamma}{12\tau}\left((\mathbf{u}^{n+1})^T\mathbf{v}^{n+1}-(\mathbf{u}^{n})^T\mathbf{v}^{n}+(\mathbf{u}^{n})^T\mathbf{v}^{n+1}-(\mathbf{u}^{n+1})^T\mathbf{v}^{n}\right)=0.
	\end{align*}
	Then, we have
	\begin{align*}
	\dfrac{\gamma}{6}(\mathbf{u}^{n+1})^T\mathbf{v}^{n+1}+\dfrac{a}{2}(\mathbf{u}^{n+1})^T\mathbf{u}^{n+1}-\dfrac{\gamma}{6}(\mathbf{u}^{n})^T\mathbf{v}^{n}+\dfrac{a}{2}(\mathbf{u}^{n})^T\mathbf{u}^{n}=0,
	\end{align*}
	which implies
	\begin{align}\label{theorem-result-eq1}
	\mathcal{H}_h^{n+1}=\mathcal{H}_h^n,\quad\forall\;n\geq 1.
	\end{align}
	Similarly, we adopt the same technique to \eqref{start-linear-scheme-eq} and obtain
	\begin{align}\label{theorem-result-eq2}
	\mathcal{H}_h^{1}=\mathcal{H}_h^0.
	\end{align}
	Combining \eqref{theorem-result-eq1} with \eqref{theorem-result-eq2} leads to \eqref{energy-conservation-law} and concludes the proof.
\end{proof}

\subsection{Linear-implicit leap-frog (LILF) scheme for the RLW equation}
Applying the linear implicit leap-frog scheme in time for \eqref{new-semi-discrete-eq}, we obtain
\begin{align}\label{linear-scheme-eq2}
\begin{split}
&(\mathbb{A}-\sigma\mathbb{B})\delta_t\mathbf{u}^n=-\mathbb{C}\left(\dfrac{\gamma}{6}A_t\mathbf{v}^{n}+aA_t\mathbf{u}^{n}\right)-\dfrac{\gamma}{3}\mathbb{C}\mathrm{diag}(\mathbf{u}^{n})A_t\mathbf{u}^{n},\\[0.3cm]
&(\mathbb{A}-\sigma\mathbb{B})\delta_t\mathbf{v}^n=-2\mathrm{diag}(\mathbf{u}^{n})\mathbb{C}\left(\dfrac{\gamma}{6}A_t\mathbf{v}^{n}+aA_t\mathbf{u}^{n}\right)-\dfrac{2}{3}\mathrm{diag(\mathbf{u}^{n})}\mathbb{C}\mathrm{diag}(\mathbf{u}^{n})A_t\mathbf{u}^{n},
\end{split}
\end{align}
where the starting value $\mathbf{u}^1$ and $\mathbf{v}^1$  are still computed by \eqref{start-linear-scheme-eq}. 
\begin{theorem}
	The scheme \eqref{linear-scheme-eq2} and \eqref{start-linear-scheme-eq} satisfy the discrete
	energy conservation law
	\begin{align}\label{energy-conservation-law2}
	\mathcal{H}_h^n=\mathcal{H}_h^0,\quad\forall\;n\geq 0,
	\end{align}
	where
	$\mathcal{H}_h^n=h\sum\limits_{j=0}^{N-1}\big(\dfrac{\gamma u^n_jv^n_j}{6}+\dfrac{a(u^n_j)^2}{2}\big).$
\end{theorem}
\begin{proof}
	The proof of this theorem is similar to the previous one and thus it is omitted.
\end{proof}

\section{Numerical experiments}
As discussed above, conservation laws play an important role in migration of the solitary. Therefore, we pay attention to the efficiency and  conservative properties of our proposed new schemes.
\noindent
The numerical error in $L^2$ and $L^{\infty}$ norms are defined by
\begin{align*}
\|E_u\|_{0,h}=\left(h\sum_{j=0}^{N-1}|u(x_j,t_n)-u_j^n|^2\right)^{1/2},\quad
\|E_u\|_{\infty}=\max_{0\leq j\leq N-1}|u(x_j,t_n)-u_j^n|.
\end{align*}
The convergence order is calculated with the formula
\begin{align*}
\mathrm{Order}=\dfrac{\log(error_1/error_2)}{\log(\delta_1/\delta_2)},
\end{align*}
where $\delta_j,\ error_j\ (j=1,2)$ are step size and the corresponding error with step size $\delta_j$, respectively.
In order to show the preservation of invariants at $n$-th time level, we test the changes in mass and energy by $\mathcal{M}_h^n-\mathcal{M}_h^0$ and $\mathcal{H}_h^n-\mathcal{H}_h^0$, respectively.

\noindent
$\mathbf{Example\; 1}$ (Motion of a single solitary wave)
The RLW equation has an analytic solution of the form
\begin{align*}
u(x,t)=3c\mathrm{sech}^2(m(x-(\gamma c+a)t-x_0)),
\end{align*}
which corresponds to the motion of a single solitary wave with amplitude $3c$, initial center at $x_0$, the wave velocity $v=a+\gamma c$ and width $m=\sqrt{\gamma c/(v\sigma)}/2$. For this problem, the theoretical values of the invariants are
\begin{align*}
M=\dfrac{6c}{m},\quad H=\dfrac{6c^2}{m}+\dfrac{24c^3}{5m},
\end{align*}
which correspond to the mass and energy \cite{ZakiSI2001}, respectively. Next, we solve
the RLW equation with initial condition
\begin{align*}
u(x,0)=3c\mathrm{sech}^2(mx).
\end{align*}
All computations are done with the parameters $x_0=0$, $a=1$, $\sigma=1$ and $\gamma=1$.

First, the proposed schemes are performed with $x\in[-60,200]$, $c=1/3$ at $T=75$. Fig. \ref{fig:error-momentum} shows the errors in mass and energy by the four proposed schemes FIEP, LIEP, LICN and LILF. The errors in mass and energy are all conserved up to roundoff error, supporting our theoretical results.

Second, we test the convergence order on the solution numerically and display the computational efficiency by considering the initial condition with $c=1$ and
the run of the algorithm is continued up to time $T=1$
over the problem domain $[-40,60]$. The accuracies of numerical solution are given in Fig. \ref{fig:accuracy-numrical-solution}, where a second order convergence can be explicitly observed. The CPU costs of the four schemes are presented in Fig. \ref{fig:CPU-time}, which shows that the schemes LIEP, LICN and LILF are more efficient than FIEP scheme.

Finally, we investigate the numerical values in mass and energy and the $L^2$, $L^{\infty}$ error norms of numerical solutions are given in Table \ref{tab:FIEP}--\ref{tab:LILF}. As it is seen from these tables, we find the numerical values of mass and energy coincide with  their analytical values. Actually, the mass and energy respectively remain almost constant and the error norms in $L^2$ and $L^{\infty}$ are satisfactorily small. In addition, we compare the errors of numerical solutions in $L^2$ and $L^{\infty}$ and the calculation time among the proposed four schemes and some existed schemes \cite{CaiHong2017, Caijiaxiang2011, Dogan2002} in Table \ref{tab:Numer-Compari}. On the one hand, our proposed implicit scheme FIEP is more efficient and effective than the existed schemes NC-II \cite{CaiHong2017} and AMC-CN \cite{Caijiaxiang2011}. On the other hand, our proposed linear schemes LIEP, LICN and LILF all show higher precision solutions than the existed scheme Linear-CN \cite{Dogan2002}.  In a word, our schemes are efficient and reliable.

\begin{figure}[!htbp]
	\centering
	\subfigure{
		\includegraphics[width=0.35\textwidth,height=0.35\textwidth]{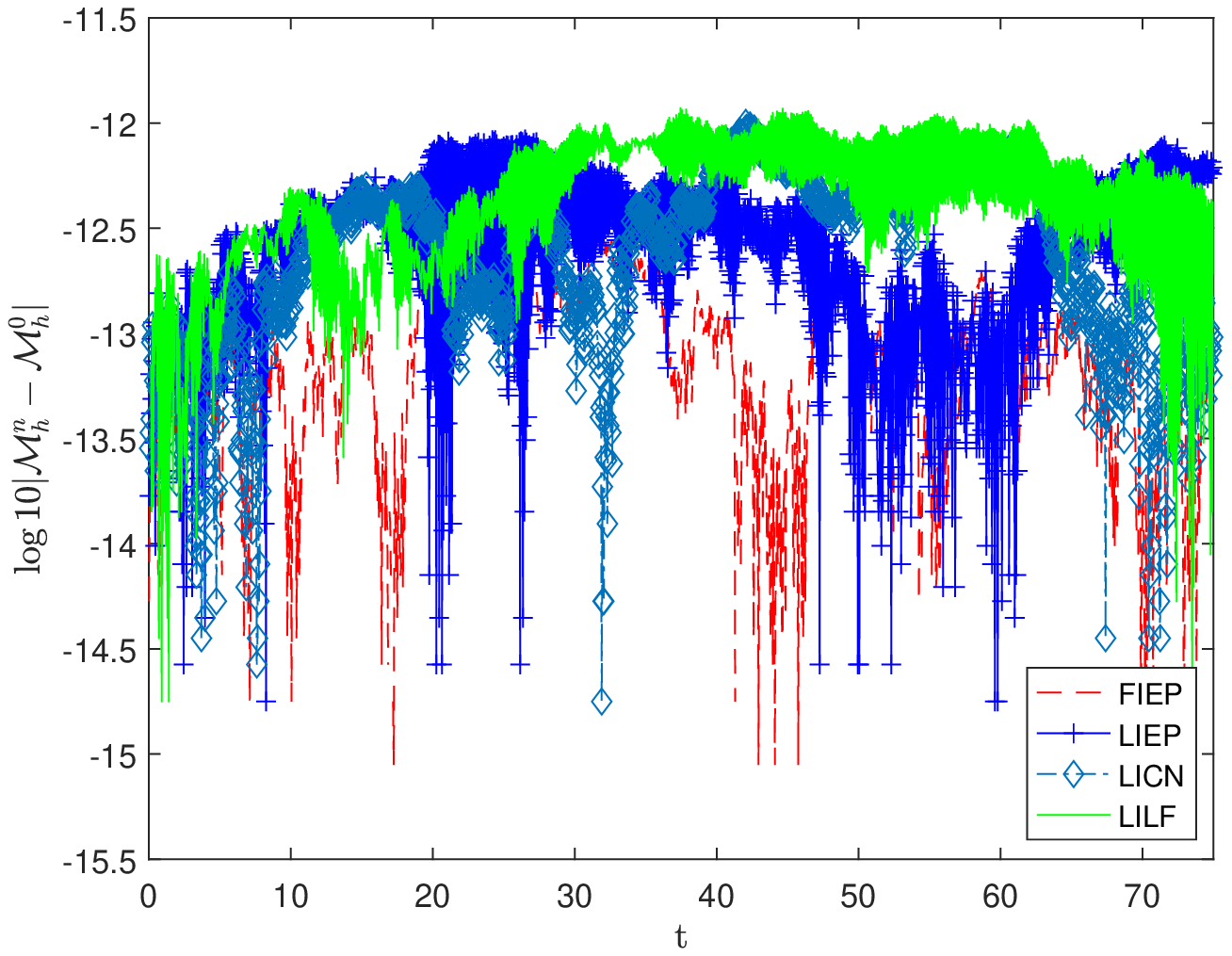}}
	\subfigure{
		\includegraphics[width=0.35\textwidth,height=0.35\textwidth]{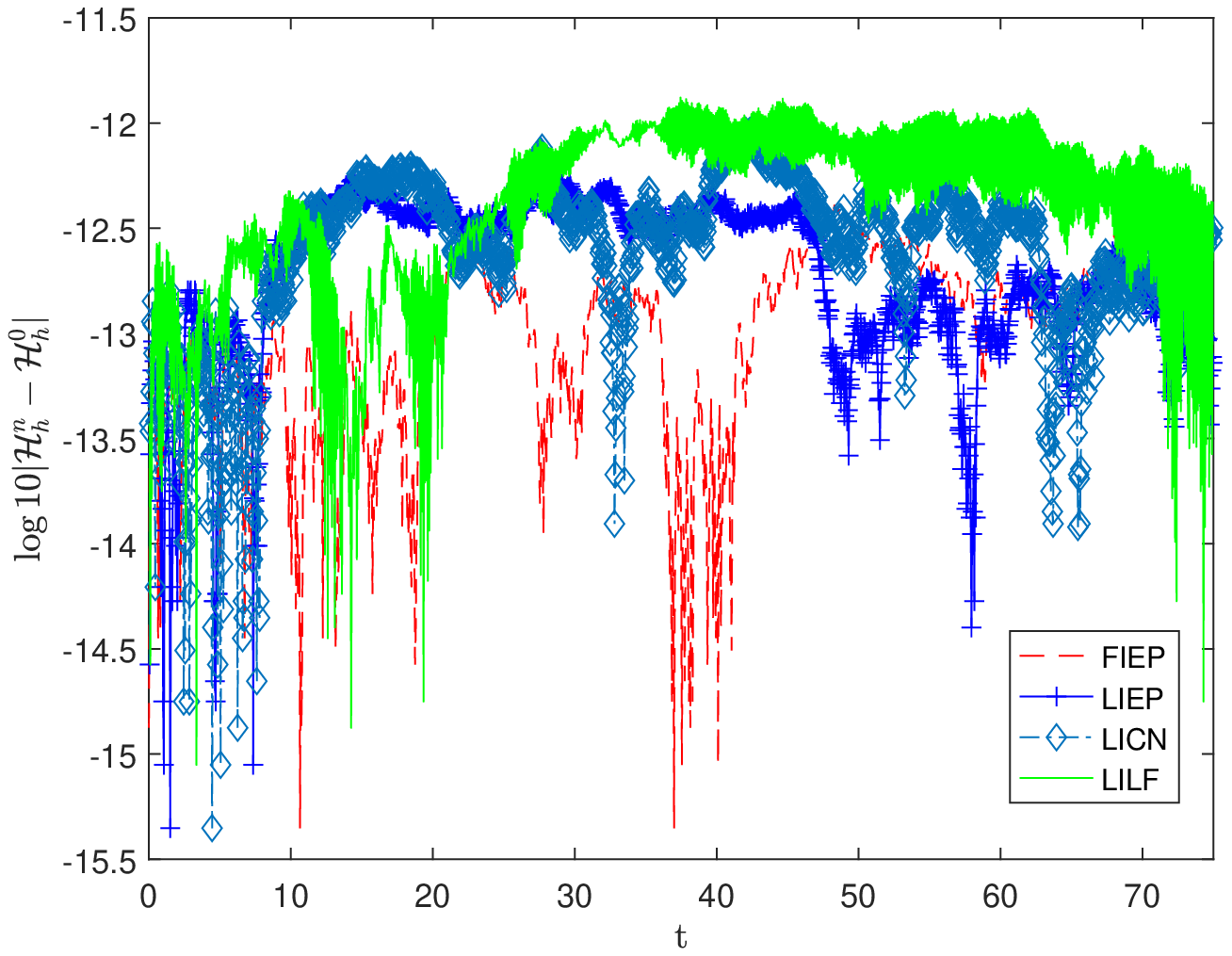}}
	\caption{\small The errors in mass (left) and energy (right) of the four schemes with $c=1/3$, $\tau=0.05$, $h=0.1$ and $x\in[-60,200]$ until $T=75$.\label{fig:error-momentum}}
\end{figure}

\begin{figure}[!htbp]
	\centering
	\subfigure{
		\includegraphics[width=0.35\textwidth,height=0.35\textwidth]{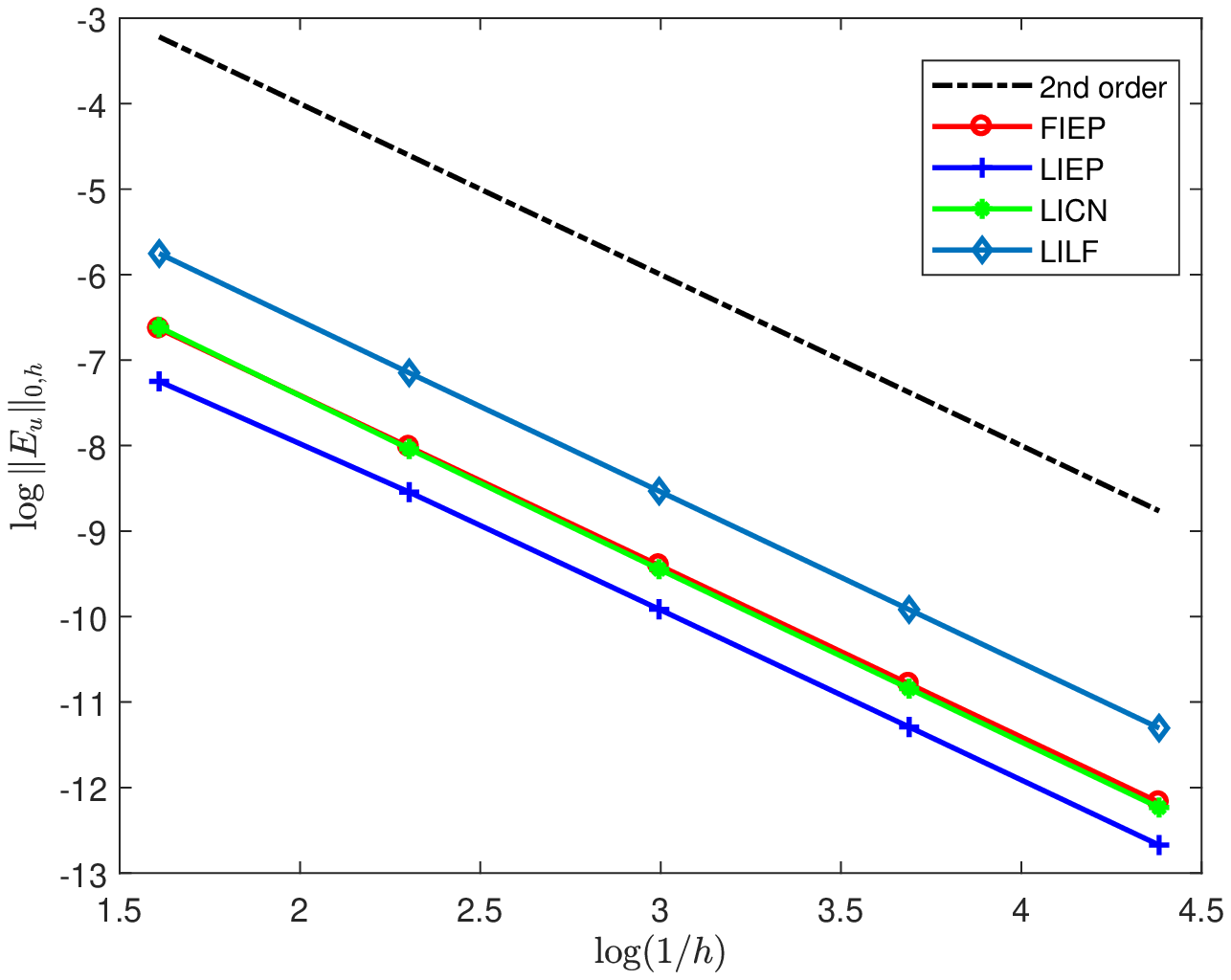}}
	\subfigure{
		\includegraphics[width=0.35\textwidth,height=0.35\textwidth]{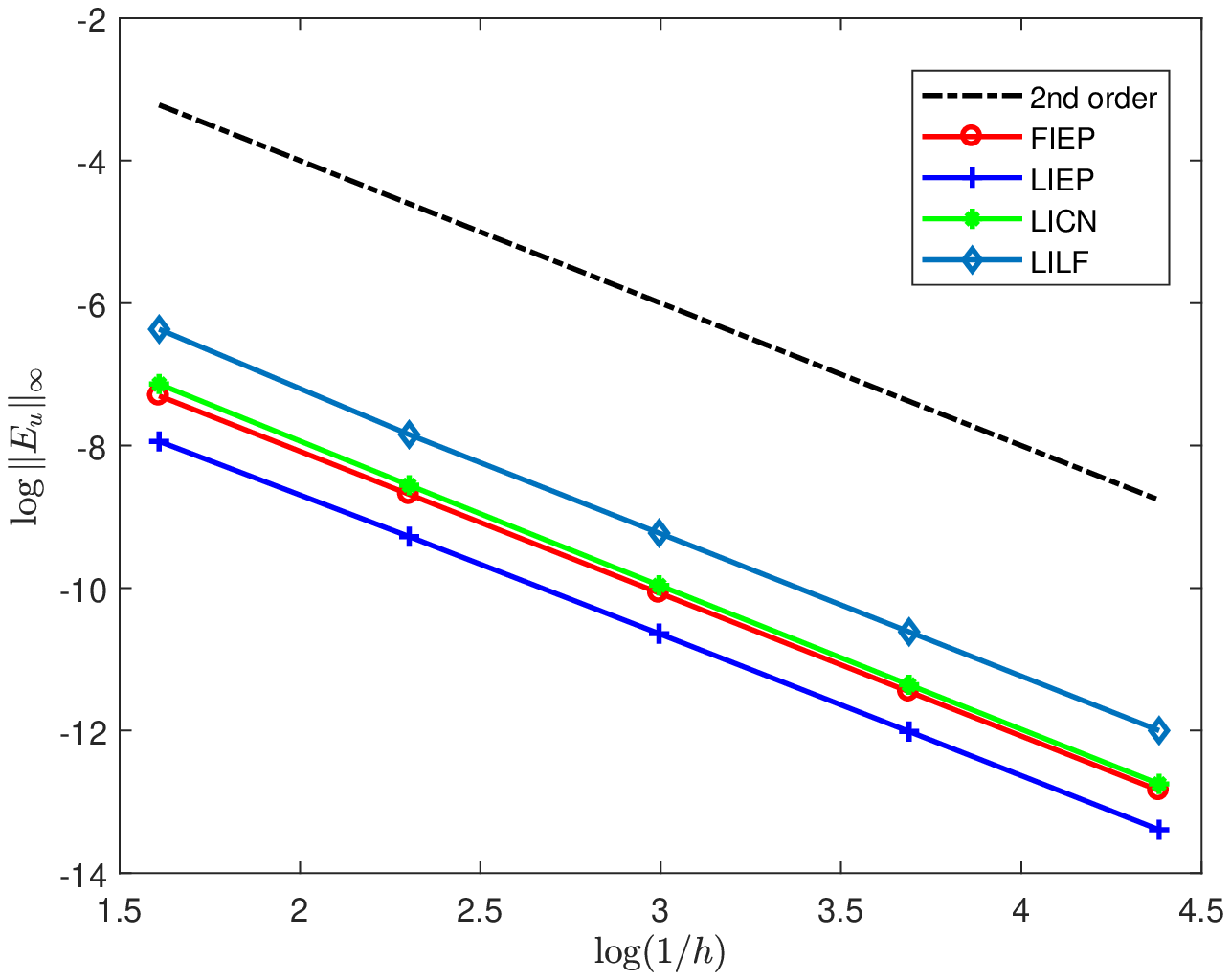}}
	\caption{\small The accuracy of numerical solutions in $L^2$ and $L^{\infty}$  errors of the four schemes with mesh size $\tau=h$\label{fig:accuracy-numrical-solution} }
\end{figure}

\begin{figure}[!htbp]
	\centering
	\subfigure{
		\includegraphics[width=0.35\textwidth,height=0.35\textwidth]{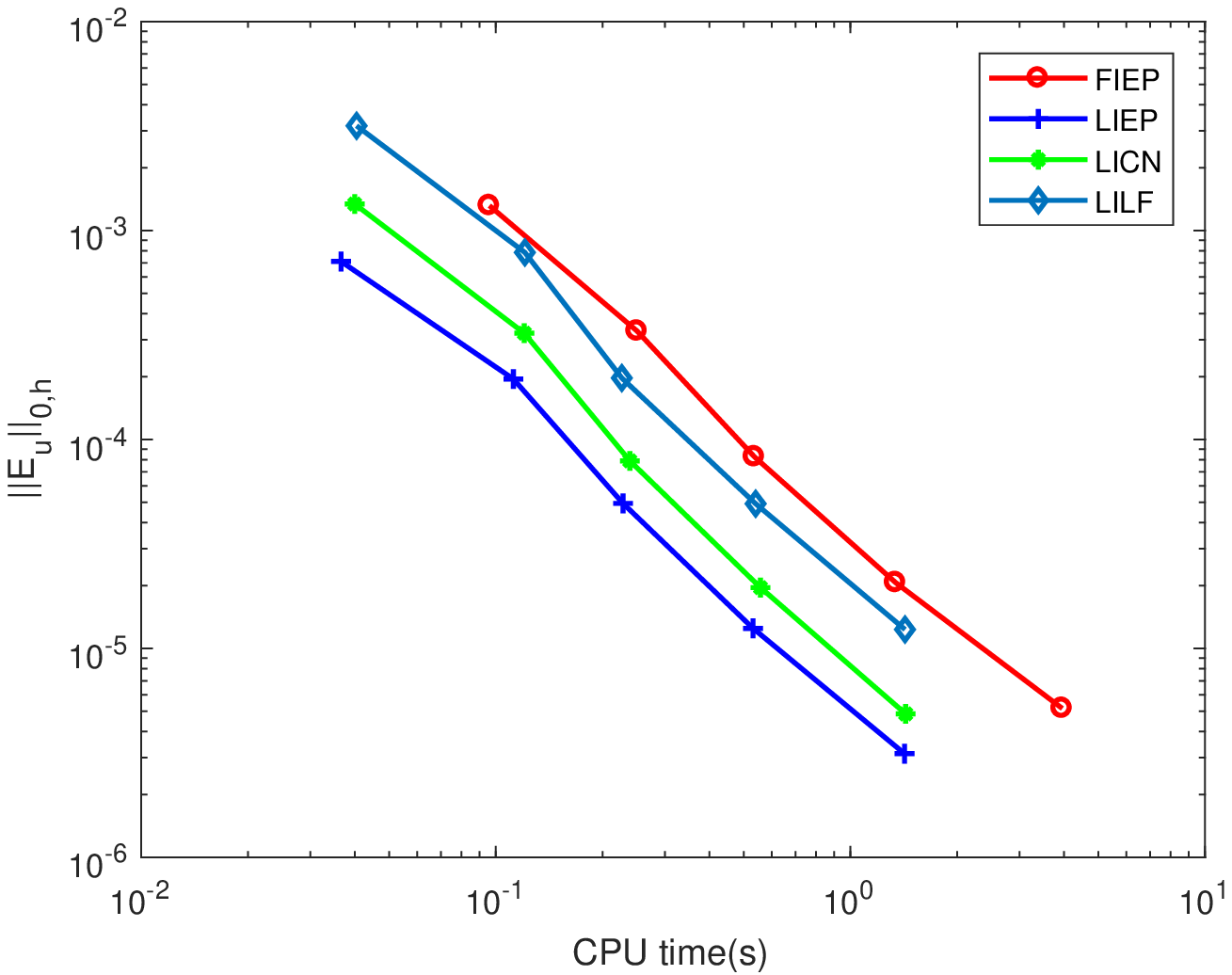}}
	\subfigure{
		\includegraphics[width=0.35\textwidth,height=0.35\textwidth]{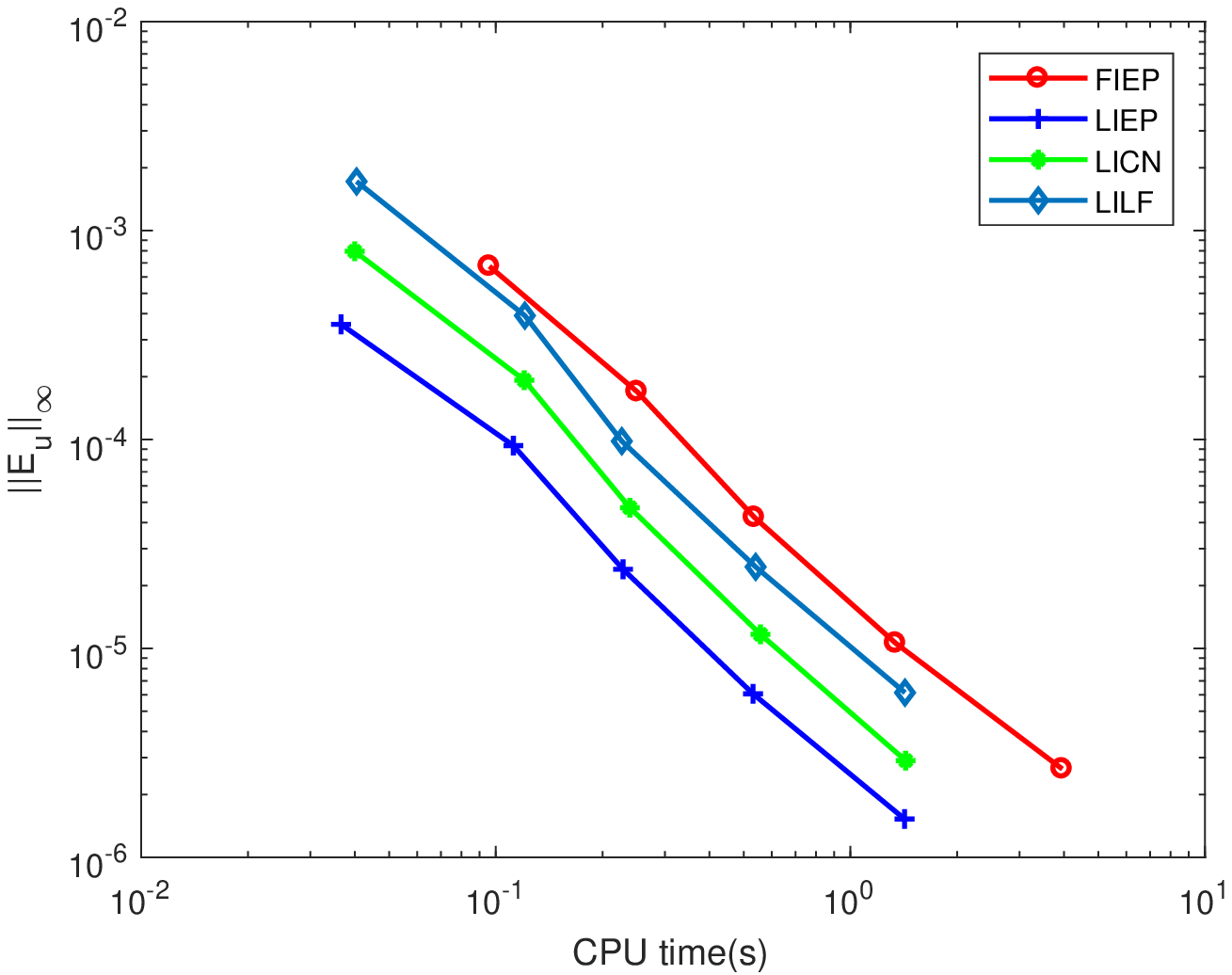}}
	\caption{\small Comparison of $L^2$ and $L^{\infty}$ errors in numerical solutions and CPU time(s) at $T=1$, where $c=1/3$ and $x\in[-40,60]$.\label{fig:CPU-time} }
\end{figure}

\begin{table}[!htbp]
	{\caption{ The invariants and errors of numerical solutions for the scheme FIEP  with $c=0.1$, $\tau=0.1$, $h=0.125$ in $[-40,60]$.}\label{tab:FIEP}}
	\begin{center}
		\begin{tabular}{c c c c c}\hline
			\small Time  &$M$       &$H$      &$L^2$ error       &$L^{\infty}$ error   \\\hline
			\small Analytical  &$3.97995$       &$0.42983$      &--       &--   \\ \hline
			0  &3.97993       &0.42983       &--       & --    \\
			4  &3.97993       &0.42983     &8.291e-5   & 3.357e-5 \\
			8  &3.97993  &0.42983     &1.633e-4  &6.721e-5      \\
			12 &3.97993  &0.42983     &2.404e-4 &9.791e-5           \\
			16&3.97993  &0.42983    &3.138e-4  &1.255e-4       \\
			\hline
		\end{tabular}
	\end{center}
\end{table}

\begin{table}[!htbp]
	{\caption{ The invariants and errors of numerical solutions for the scheme LIEP  with $c=0.1$, $\tau=0.1$, $h=0.125$ in $[-40,60]$.}\label{tab:LIEP}}
	\begin{center}
		\begin{tabular}{c c c c c}\hline
			\small Time  &$M$       &$H$      &$L^2$ error       &$L^{\infty}$ error   \\\hline
			\small Analytical  &$3.97995$   &$0.42983$      &--       &--   \\ \hline
			0  &3.97993  &0.42979  &--      & --    \\
			4  &3.97993  &0.42979  &4.020e-5 &1.455e-5      \\
			8  &3.97993  &0.42979  &8.265e-5 &3.124e-5      \\
			12 &3.97993  &0.42979  &1.224e-4 &4.673e-5       \\
			16 &3.97993  &0.42979  &1.614e-4 &6.131e-5       \\
			\hline
		\end{tabular}
	\end{center}
\end{table}\label{Tab:table-2}

\begin{table}[!htbp]
	{\caption{ The invariants and errors of numerical solutions for the scheme LICN  with $c=0.1$, $\tau=0.1$, $h=0.125$ in $[-40,60]$.}\label{tab:LICN}}
	\begin{center}
		\begin{tabular}{c c c c c}\hline
			\small Time  &$M$       &$H$      &$L^2$ error       &$L^{\infty}$ error   \\\hline
			\small Analytical  &$3.97995$   &$0.42983$      &--       &--   \\ \hline
			0  &3.97993  &0.42983  &--      & --    \\
			4  &3.97993  &0.42983  &6.485e-5 &2.651e-5      \\
			8  &3.97993  &0.42983  &1.270e-4 &5.174e-5      \\
			12 &3.97993  &0.42983  &1.854e-4 &7.413e-5      \\
			16 &3.97993  &0.42983  &2.416e-4 &9.502e-5      \\
			\hline
		\end{tabular}
	\end{center}
\end{table}

\begin{table}[!htbp]
	{\caption{ The invariants and errors of numerical solutions for the scheme LILF  with $c=0.1$, $\tau=0.1$, $h=0.125$ in $[-40,60]$.}\label{tab:LILF}}
	\begin{center}
		\begin{tabular}{c c c c c}\hline
			\small Time  &$M$       &$H$      &$L^2$ error       &$L^{\infty}$ error   \\\hline
			\small Analytical  &$3.97995$   &$0.42983$      &--       &--   \\ \hline
			0  &3.97993  &0.42983  &--      & --    \\
			4  &3.97993  &0.42983  &1.998e-4 &7.882e-5      \\
			8  &3.97993  &0.42983  &3.936e-4 &1.574e-4      \\
			12 &3.97993  &0.42983  &5.842e-4 &2.332e-4       \\
			16 &3.97993  &0.42983  &7.671e-4 &3.021e-4      \\
			\hline
		\end{tabular}
	\end{center}
\end{table}

\begin{table}[!htbp]
	{\caption{Numerical comparison at $T=10$ with $c=0.1$, $\tau=0.1$ and $-40\leq x\leq 60$.}\label{tab:Numer-Compari}}
	\begin{center}
		\begin{tabular}{l l l l l l l}\hline
			\multirow{2}{*}{Method} &\multicolumn{3}{c}{$h=0.125$} &\multicolumn{3}{c}{$h=0.0625$}\\ \cmidrule {2-4} \cmidrule{5-7}
			&$L^2$ error       &$L^{\infty}$ error &CPU(s)  &$L^2$ error       &$L^{\infty}$ error  &CPU(s) \\ \hline
			FIEP&2.023e-4&8.298e-5&1.398&1.363e-4&5.520e-5&1.796      \\
			LIEP&1.035e-4&3.946e-5&0.993&1.687e-4&6.716e-5&1.268         \\
			LICN&1.566e-4&6.316e-5&1.073&9.172e-5&3.545e-5&1.076         \\
			LILF&4.896e-4&1.962e-4&1.006&4.241e-4&1.684e-4&1.128      \\
			NC-II \cite{CaiHong2017}   &2.088e-4&7.532e-5&1.755&1.234e-4&4.236e-5&2.137 \\
			AMC-CN \cite{Caijiaxiang2011}    &3.944e-4&1.581e-4&1.506&1.828e-4&7.307e-5&1.988\\
			Linear-CN \cite{Dogan2002} &2.648e-4&1.088e-4&1.046&7.945e-4&2.957e-4&1.108    \\
			\hline
		\end{tabular}
	\end{center}
\end{table}

\noindent
$\mathbf{Example\; 2}$ (Interaction of three positive solitary waves) This example shows the interaction of three solitary waves with different amplitudes and travelling in the same direction for the RLW equation with parameters $\gamma=1$, $\sigma=1$ and $a=1$. We consider the initial condition $u(x,0)=\sum_{i=1}^33c_i\mathrm{sech}^2(m_i(x-x_i))$, where $-200\leq x\leq400$, $c_1=1$, $c_2=0.5$, $c_3=0.25$, $x_1=-20$, $x_2=15$, $x_3=45$ and $m_i=\frac{1}{2}\sqrt{\frac{\gamma c_i}{(1+\gamma c_i)\sigma}}$. The simulation is performed with $h=0.25$ and $\tau=0.05$ until $T=400$. Due to the space restrictions, we just display the interaction of three solitary waves as time evolves using the scheme LICN. Fig. \ref{fig:inter-LICN} shows the process of interaction of three positive solitary waves as time evolves. It is clear that the three solitary waves travel forward, then interact, and finally depart without
any changes in their own shapes. We also plot the errors in mass and energy in Fig. \ref{fig:mass-energy}. These results show that the errors of mass and energy computed by the schemes LIEP, LICN, LILF are conserved up to roundoff error very well than FIEP throughout the interaction simulation.
\begin{figure}[!htbp]
	\centering
	\subfigure{
		\includegraphics[width=0.3\textwidth,height=0.3\textwidth]{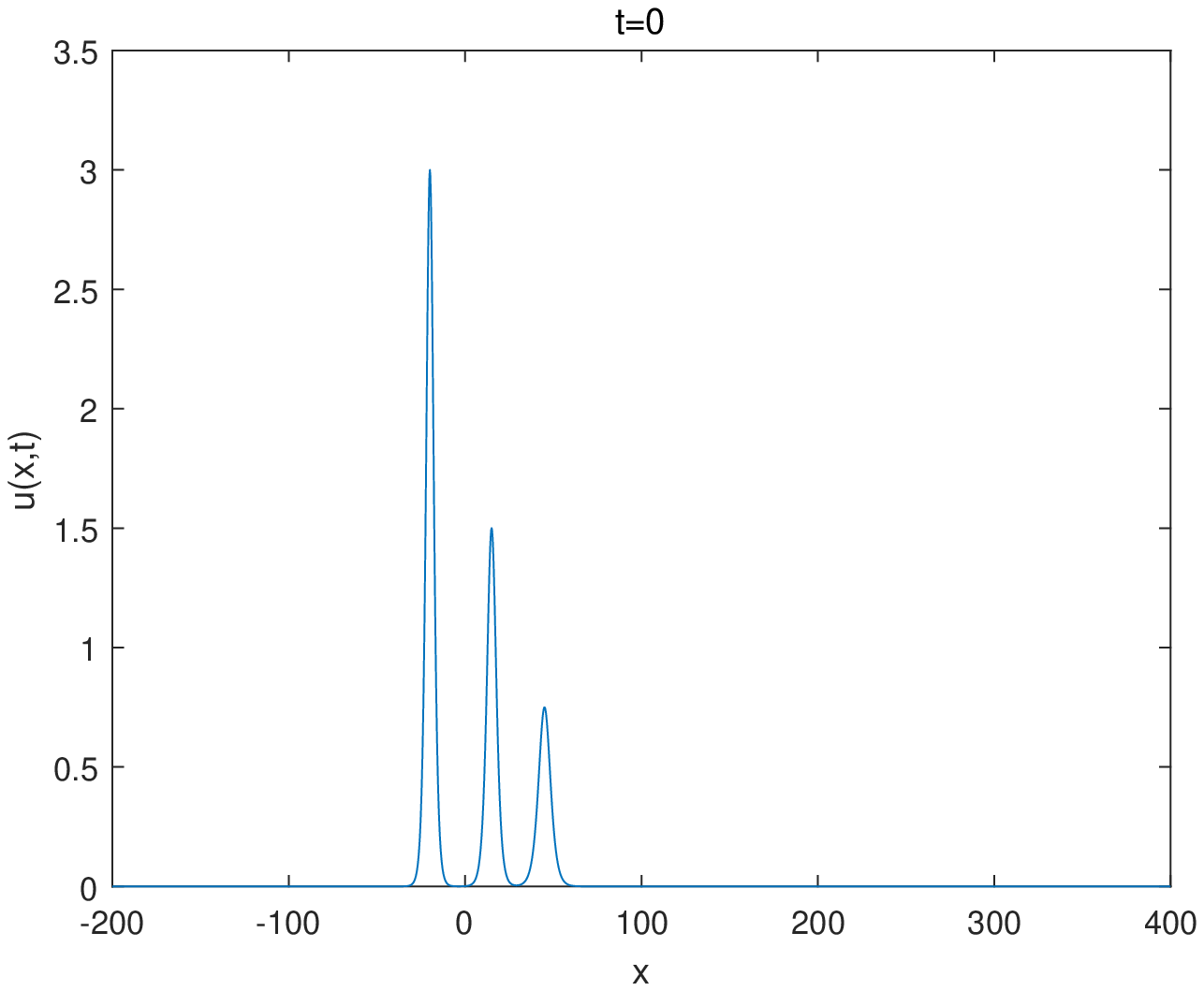}}
	\subfigure{
		\includegraphics[width=0.3\textwidth,height=0.3\textwidth]{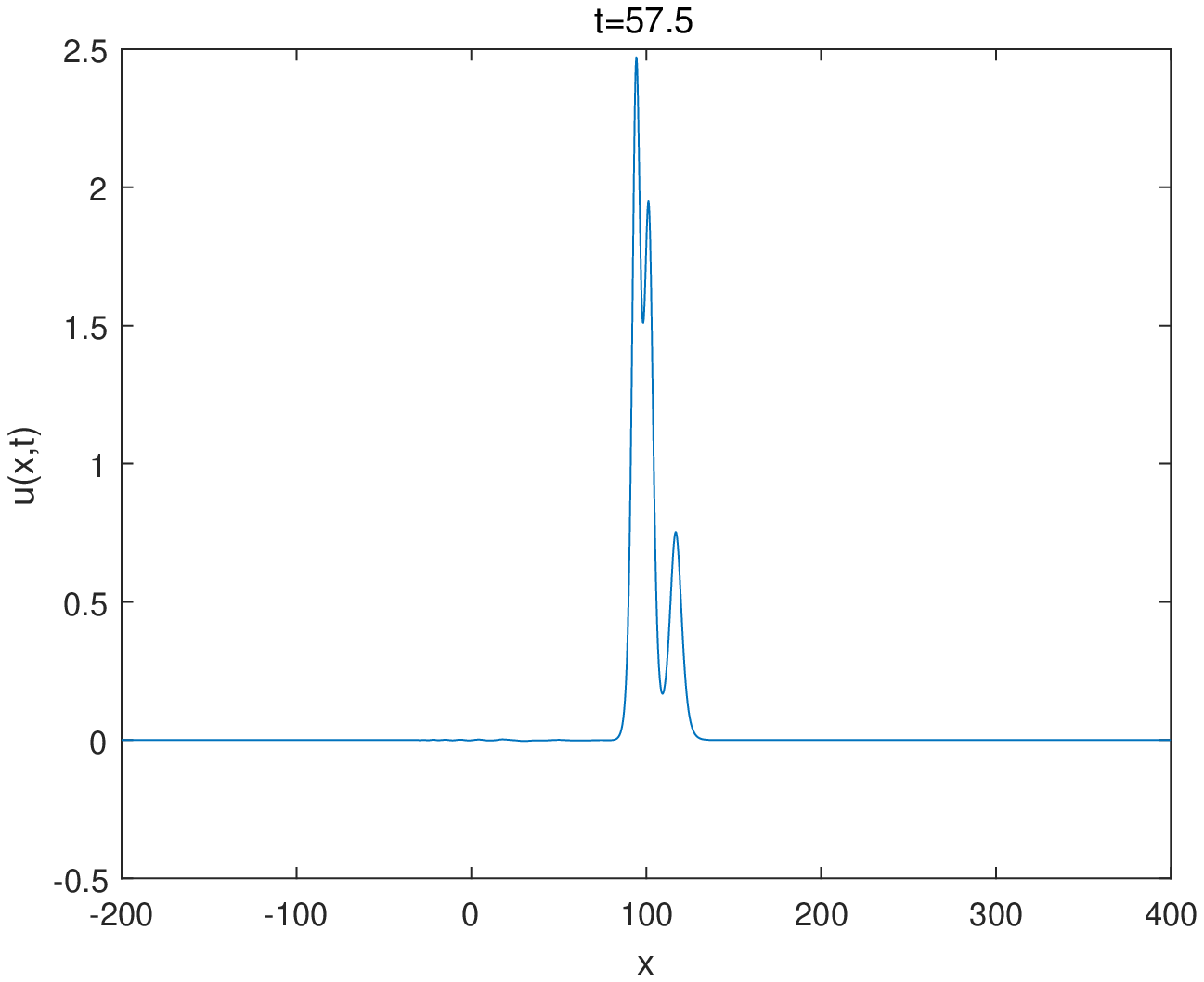}}
	\subfigure{
		\includegraphics[width=0.3\textwidth,height=0.3\textwidth]{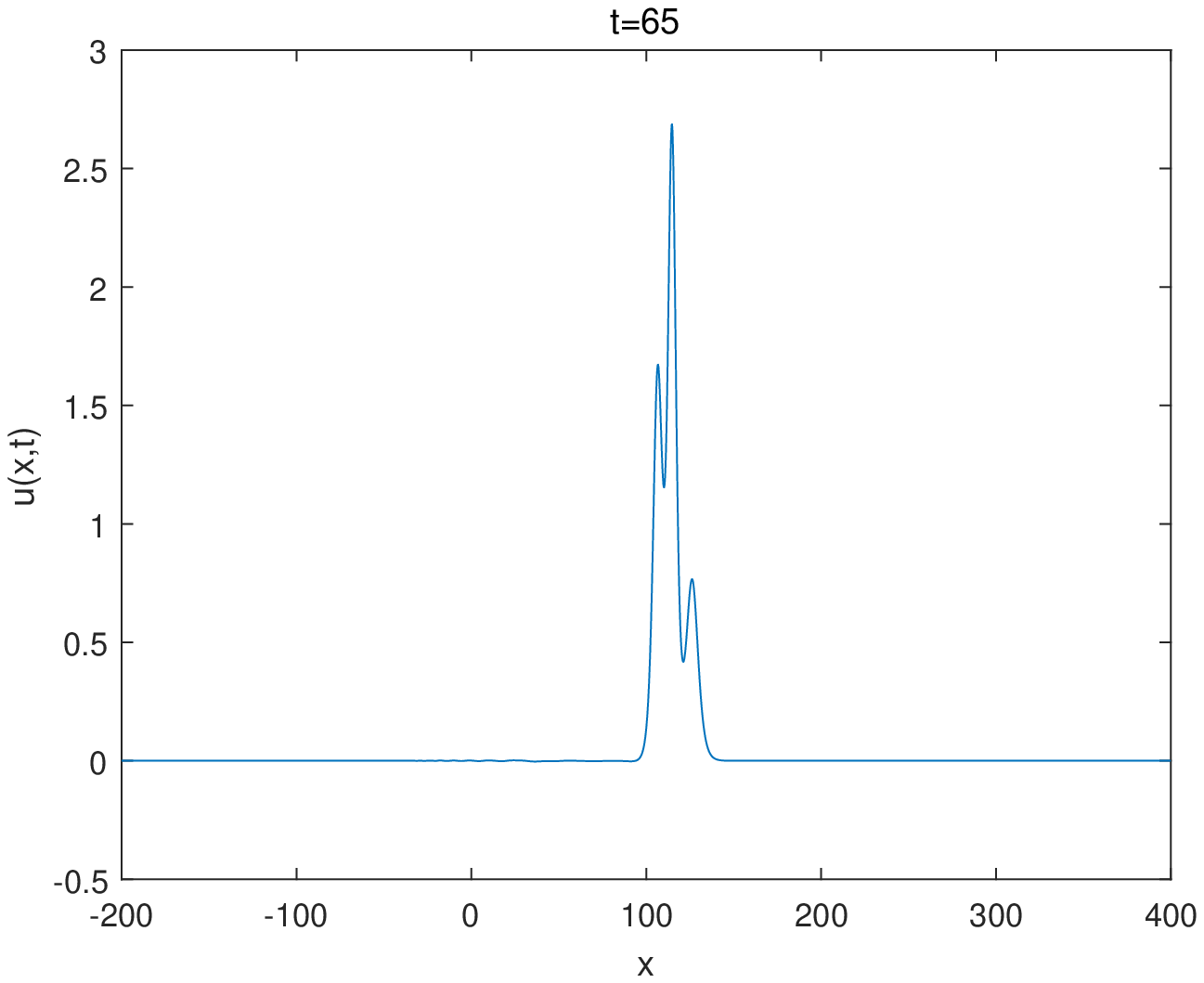}}\\
	\subfigure{
		\includegraphics[width=0.3\textwidth,height=0.3\textwidth]{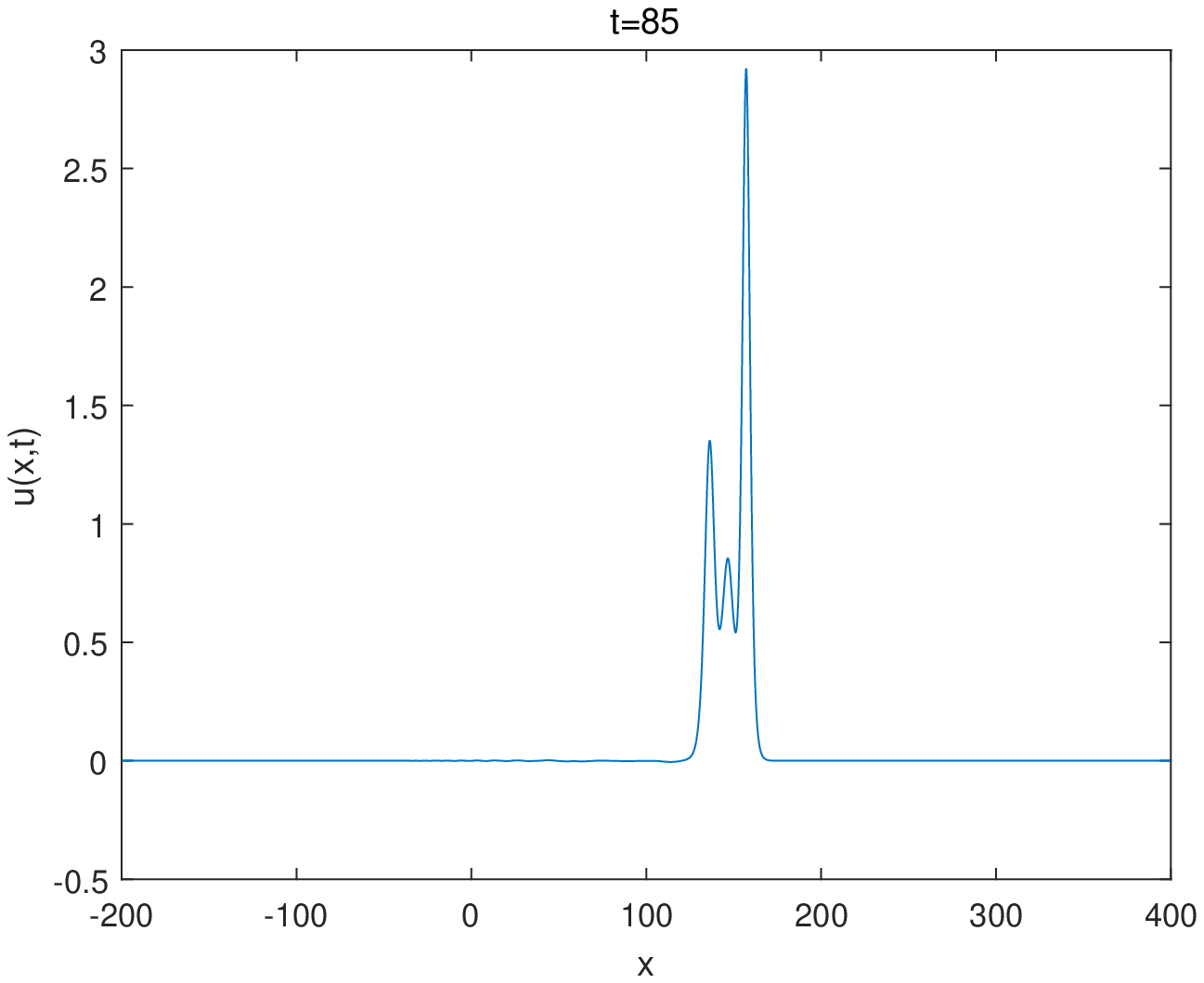}}
	\subfigure{
		\includegraphics[width=0.3\textwidth,height=0.3\textwidth]{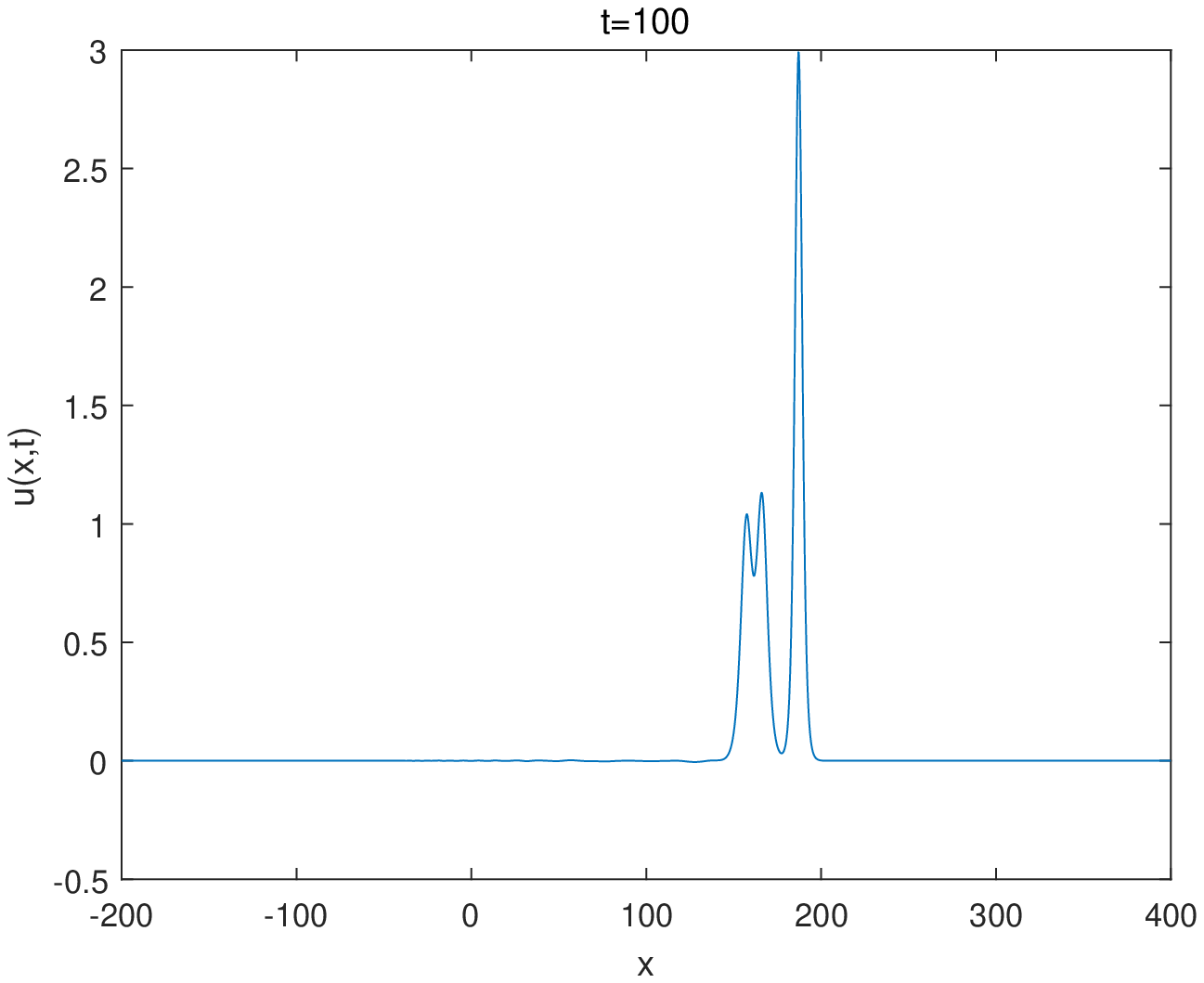}}
	\subfigure{
		\includegraphics[width=0.3\textwidth,height=0.3\textwidth]{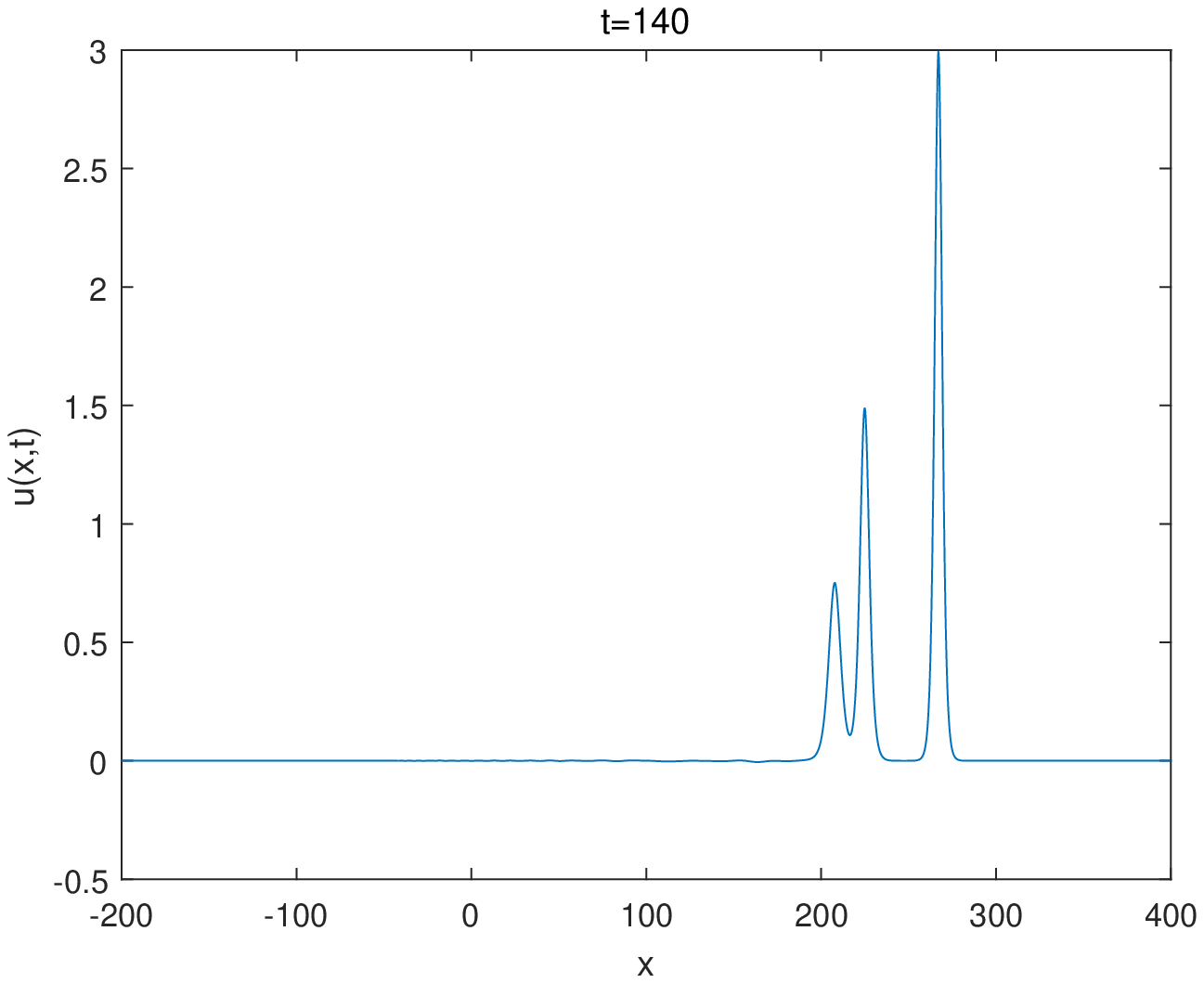}}
	\caption{\small The interaction of three solitary waves at different time using the scheme LICN.\label{fig:inter-LICN}}
\end{figure}

\begin{figure}[!htbp]
	\centering
	\subfigure{
		\includegraphics[width=0.35\textwidth,height=0.35\textwidth]{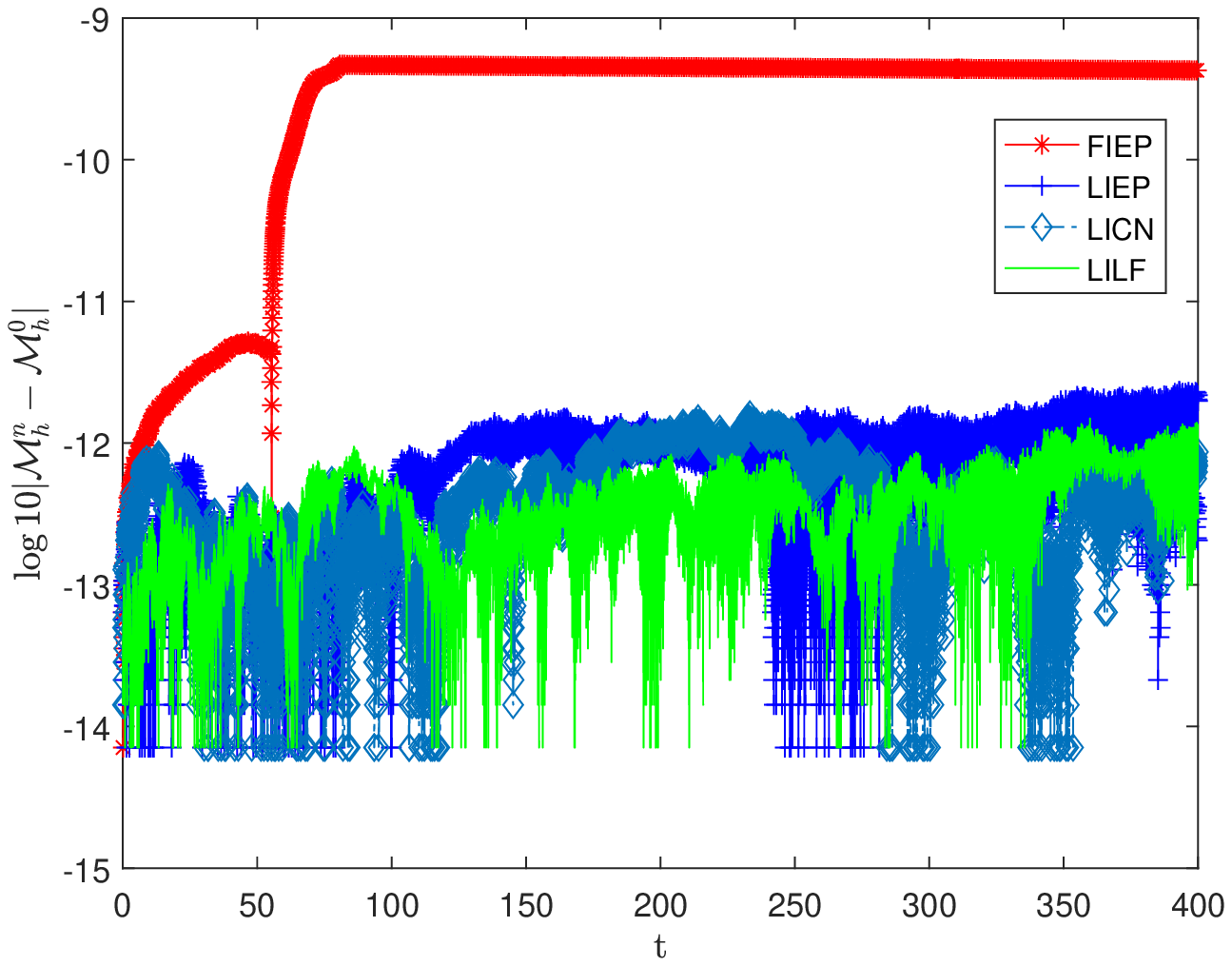}}
	\subfigure{
		\includegraphics[width=0.35\textwidth,height=0.35\textwidth]{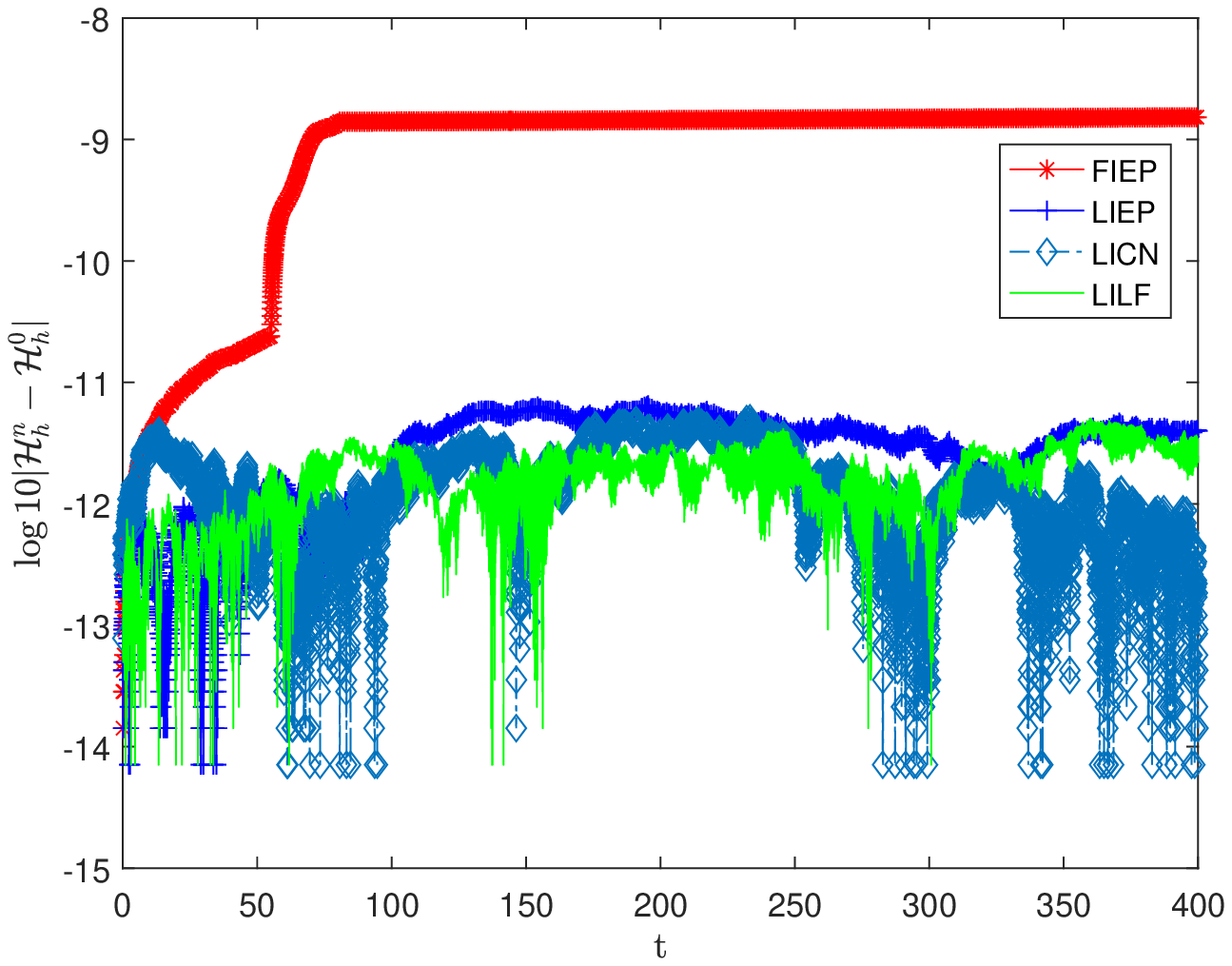}}
	\caption{\small The errors in mass (left) and energy (right) of the four schemes with $\tau=0.05$, $h=0.25$ and $x\in[-200,400]$ until $T=400$.\label{fig:mass-energy} }
\end{figure}

\noindent
$\textbf{Example\; 3}$ (The Maxwellian pulse) In this part, we have examined the evolution of an initial Maxwellian pulse into solitary waves for various values of the parameter $\delta$. Take the initial condition
\begin{align*}
u(x,0)=\exp(-(x-7)^2),\quad -40\leq x\leq 100.
\end{align*}
and all simulations are done with $\gamma=1$, $a=1$, $\tau=0.05$ and $h=0.05$. We mainly discuss each of the following cases: (i) $\sigma=0.04$, (ii) $\sigma=0.01$ and (iii) $\sigma=0.001$, respectively. The evolutions of the RLW equation at $T=55$ are just simulated by the scheme LILF due to the limit of page. For case (i), only one solitary wave is generated as shown in Fig. \ref{fig:evol-LILF} (a), for case (ii), three stable solitary waves are generated as shown in Fig. \ref{fig:evol-LILF} (b) and for case (iii), the Maxwellian initial condition has decayed into about six solitary waves as shown in Fig. \ref{fig:evol-LILF} (c). The errors in mass and energy are shown in Fig. \ref{fig:mass-energy-3}. The results imply the mass and energy are captured exactly by the schemes LIEP, LICN and LILF than FIEP.
\begin{figure}[!htbp]
	\centering
	\subfigure[]{
		\includegraphics[width=0.3\textwidth,height=0.3\textwidth]{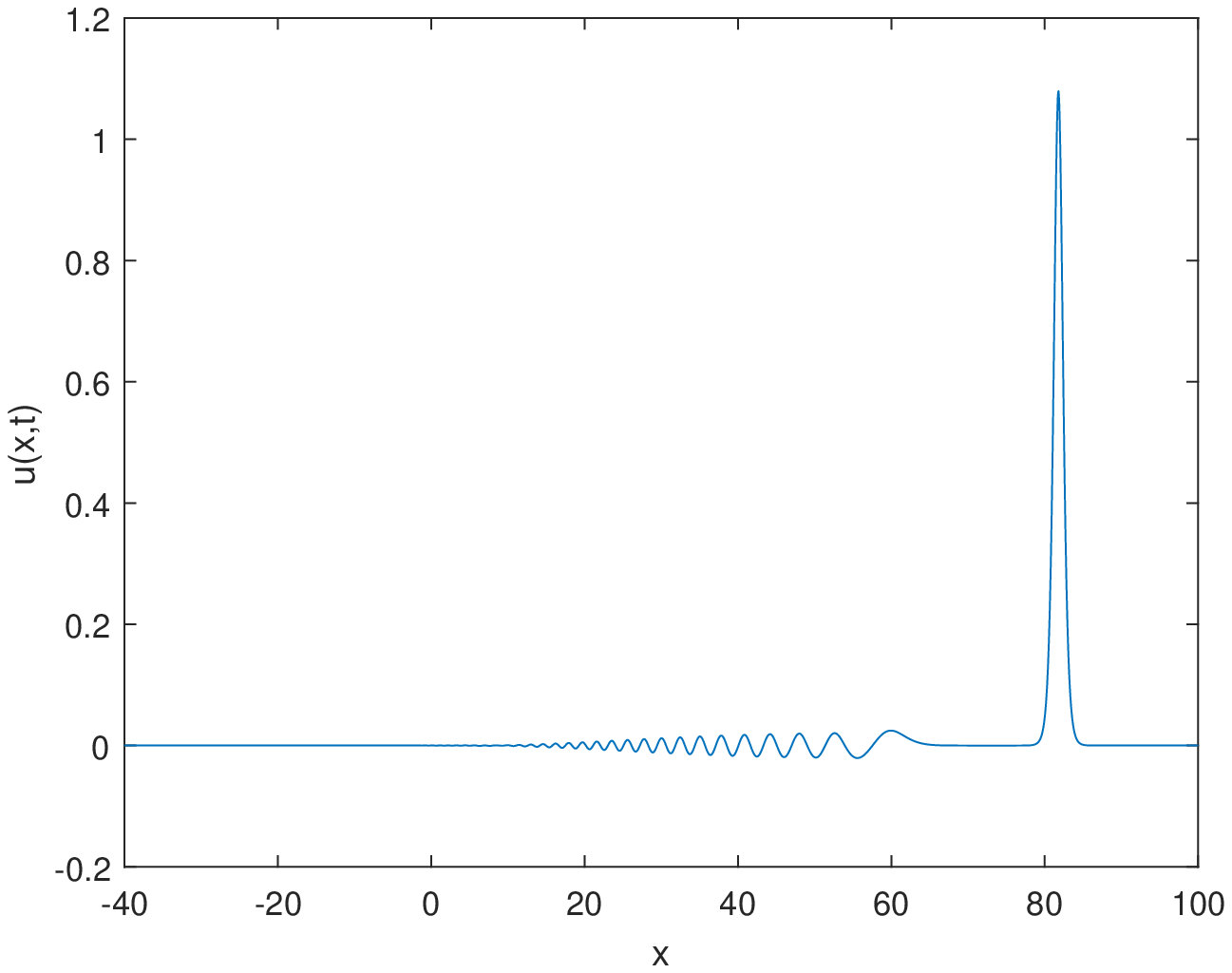}}
	\subfigure[]{
		\includegraphics[width=0.3\textwidth,height=0.3\textwidth]{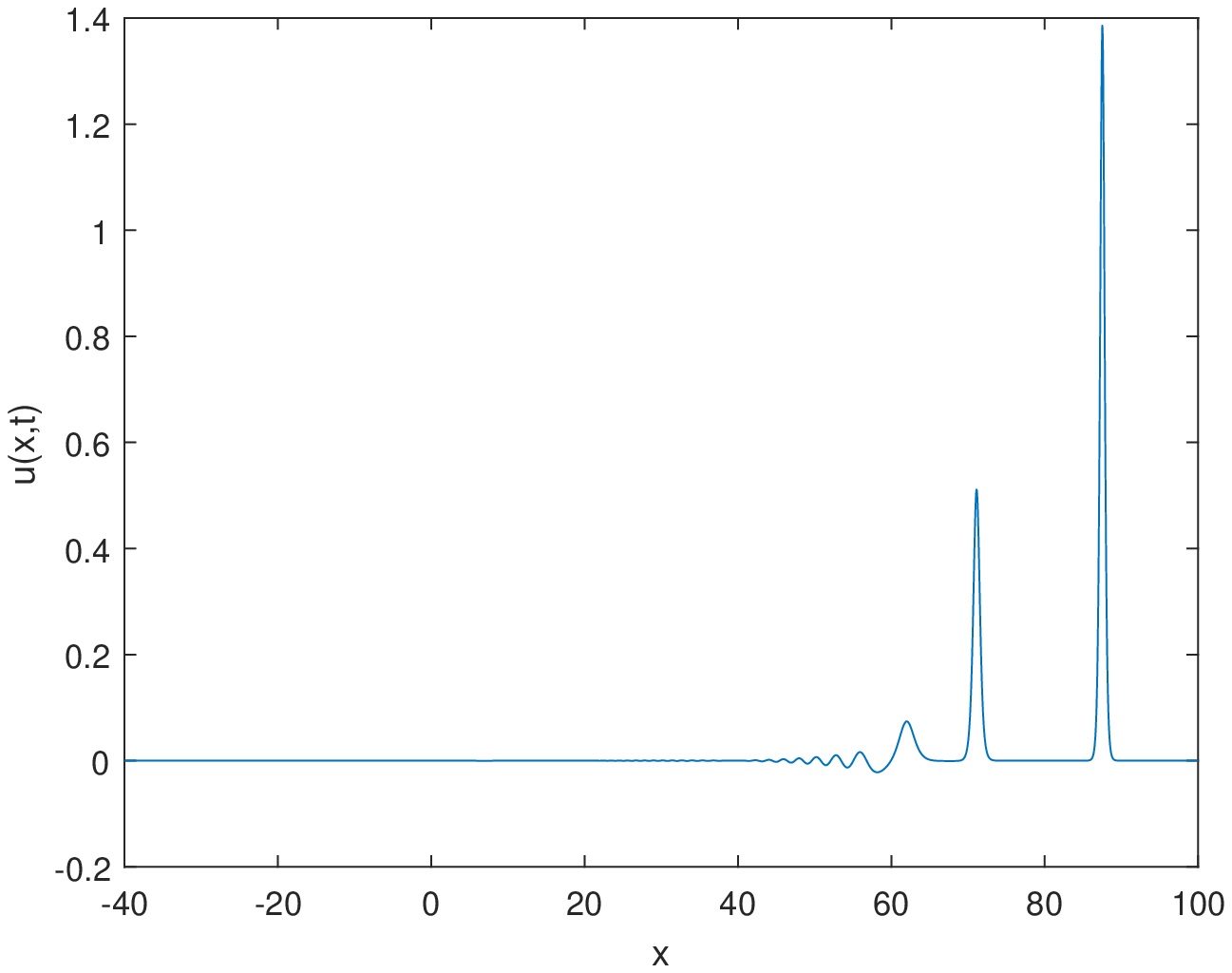}}
	\subfigure[]{
		\includegraphics[width=0.3\textwidth,height=0.3\textwidth]{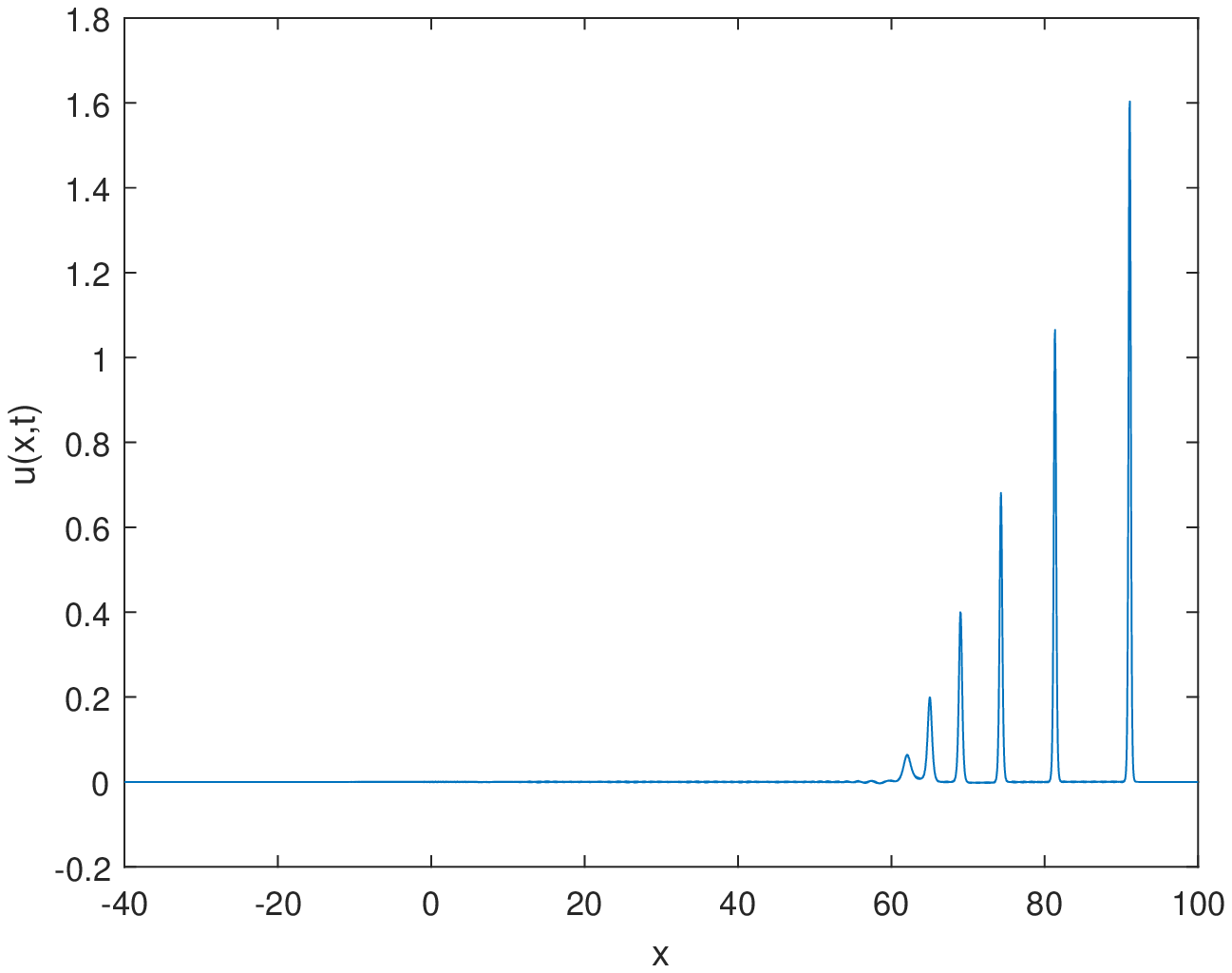}}
	\caption{\small The evolution of the RLW equation using the scheme LILF at $T=55$.\label{fig:evol-LILF} }
\end{figure}

\begin{figure}[!htbp]
	\centering
	\subfigure{
		\includegraphics[width=0.35\textwidth,height=0.35\textwidth]{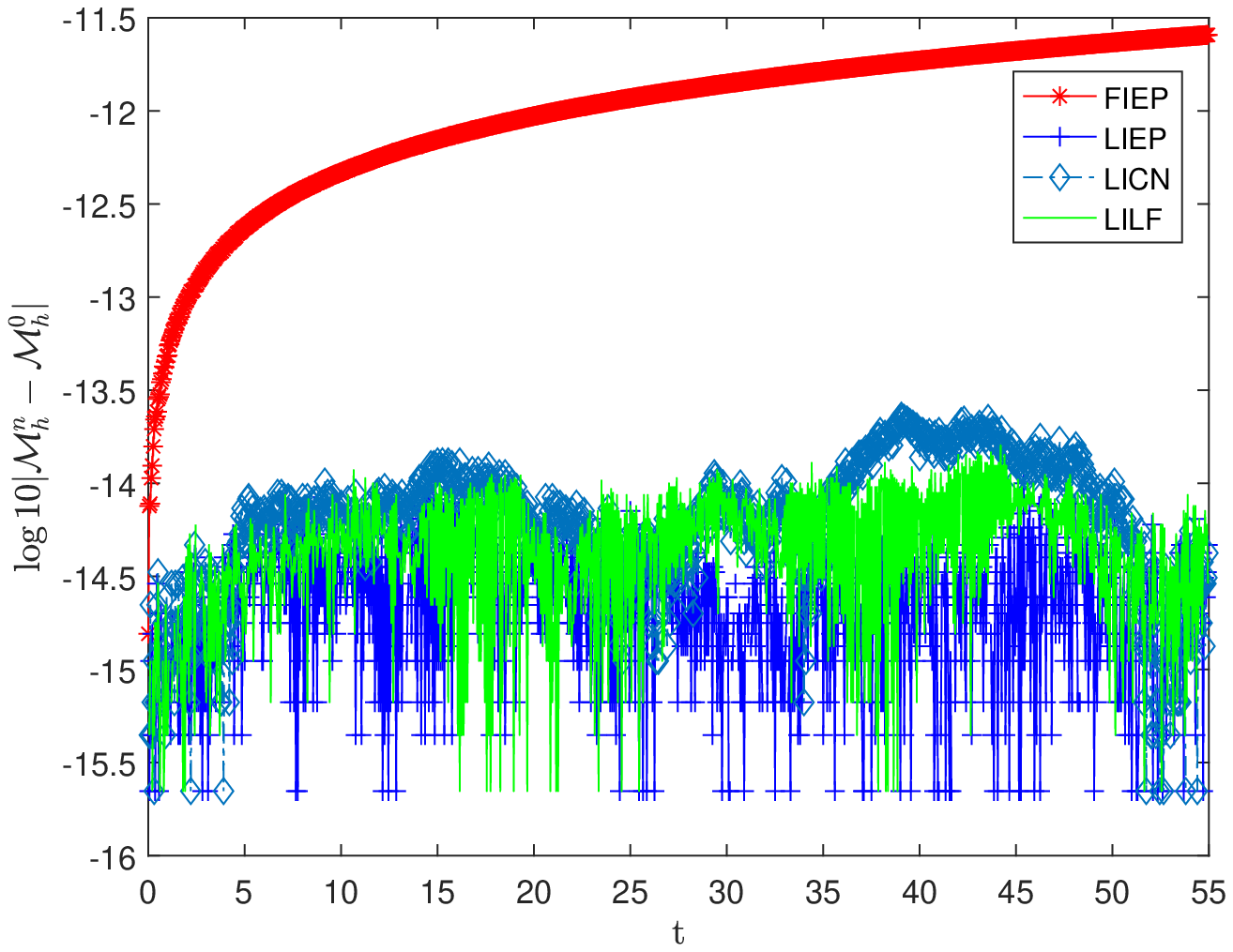}}
	\subfigure{
		\includegraphics[width=0.35\textwidth,height=0.35\textwidth]{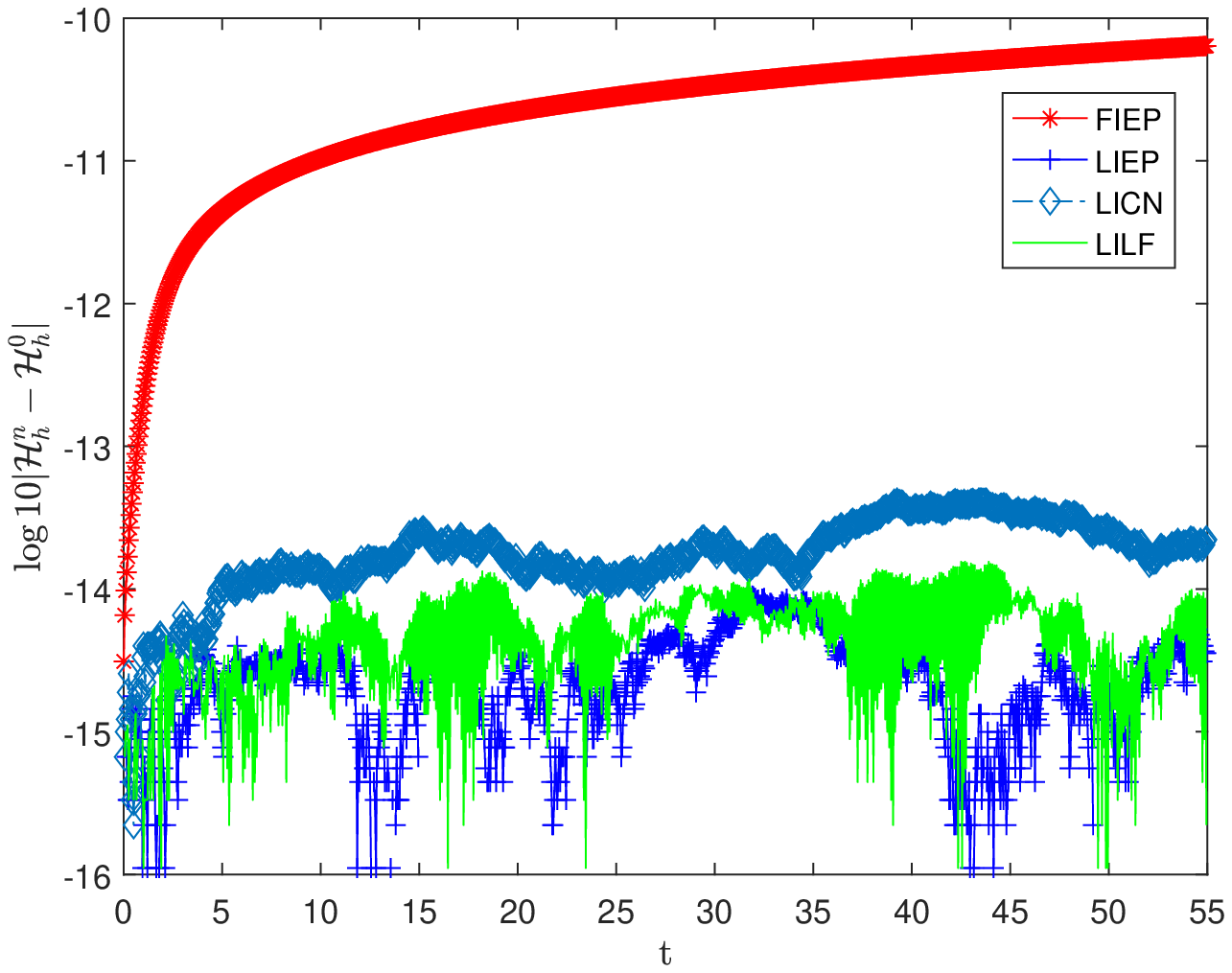}}
	\caption{\small The errors in mass (left) and energy (right) of the four schemes with $\sigma=0.01$ $\tau=0.05$ and $h=0.05$ and $x\in[-40,100]$ until $T=55$.\label{fig:mass-energy-3}}
\end{figure}

\noindent
$\mathbf{Example\; 4}$ (The undular bore propagation) As our last test problem, we consider the development of an undular bore with the initial condition
\begin{align*}
u(x,0)=\dfrac{U_0}{2}\left[1-\tanh(\dfrac{x-x_0}{d})\right],
\end{align*}
and boundary conditions
\begin{align*}
u(a,t)=U_0,\quad u(b,t)=0,
\end{align*}
where $u(x,0)$ denotes the elevation of the water above the equilibrium surface at time $t=0$, $U_0$ represents the magnitude of the change in water level which is centered on $x=x_0$ and $d$ represents the slope between the still water and deeper water. Under the above physical boundary conditions, the mass and energy are not constants but increase linearly throughout the simulation at the following rates \cite{ZakiSI2001}
\begin{align}\label{linear-behaviour-eq}
\begin{split}
M_1&=\dfrac{d}{dt}M=\dfrac{d}{dt}\int udx=U_0+\dfrac{1}{2}U_0^2,\\[0.3cm]
M_3&=\dfrac{d}{dt}H=\dfrac{d}{dt}\int\left(\dfrac{\gamma}{6}u^3+\dfrac{1}{2}u^2\right)dx=\dfrac{1}{2}U_0^2+\dfrac{\gamma}{2}U_0^3+\dfrac{\gamma}{8}U_0^4.
\end{split}
\end{align}

For the simulation, computations use the parameters $\sigma=1/6$, $\gamma=1.5$, $a=1$, $x_0=0$, $h=0.24$, $\tau=0.1$ and $U_0=0.1$. The development of the undular bore at different times from $T=0$ to $T=250$ for $d=2$ and $d=5$ are represented in Fig. \ref{fig:d2-d5}, respectively. It is observed that both waves are stable without any numerical perturbations. The rate of the growth of amplitudes of the undulations seems to be fast in the beginning. The reason is that
the formation of the undular bore depends on the form of the initial undulation.
To see the effect of the initial undulation, view
of the formation of the leading undulation for the both steep and gentle slope is illustrated by drawing maximal $u$ versus time $t$ in Fig. \ref{fig:mass-energy-d2-d5} (a). As it seen from this picture, we find the magnitudes of the leading undulations becomes very close after certain time. For the steep slope, rate of the
growth undulation is sharp and then decrease slowly. Fig. \ref{fig:mass-energy-d2-d5} (b) and (c) show the linear behaviour of the mass and energy in time for $d=2$ and $d=5$, respectively. We find that these results are consistent with \eqref{linear-behaviour-eq} very well.\\

\begin{figure}[!htbp]
	\centering
	\subfigure{
		\includegraphics[width=0.3\textwidth,height=0.3\textwidth]{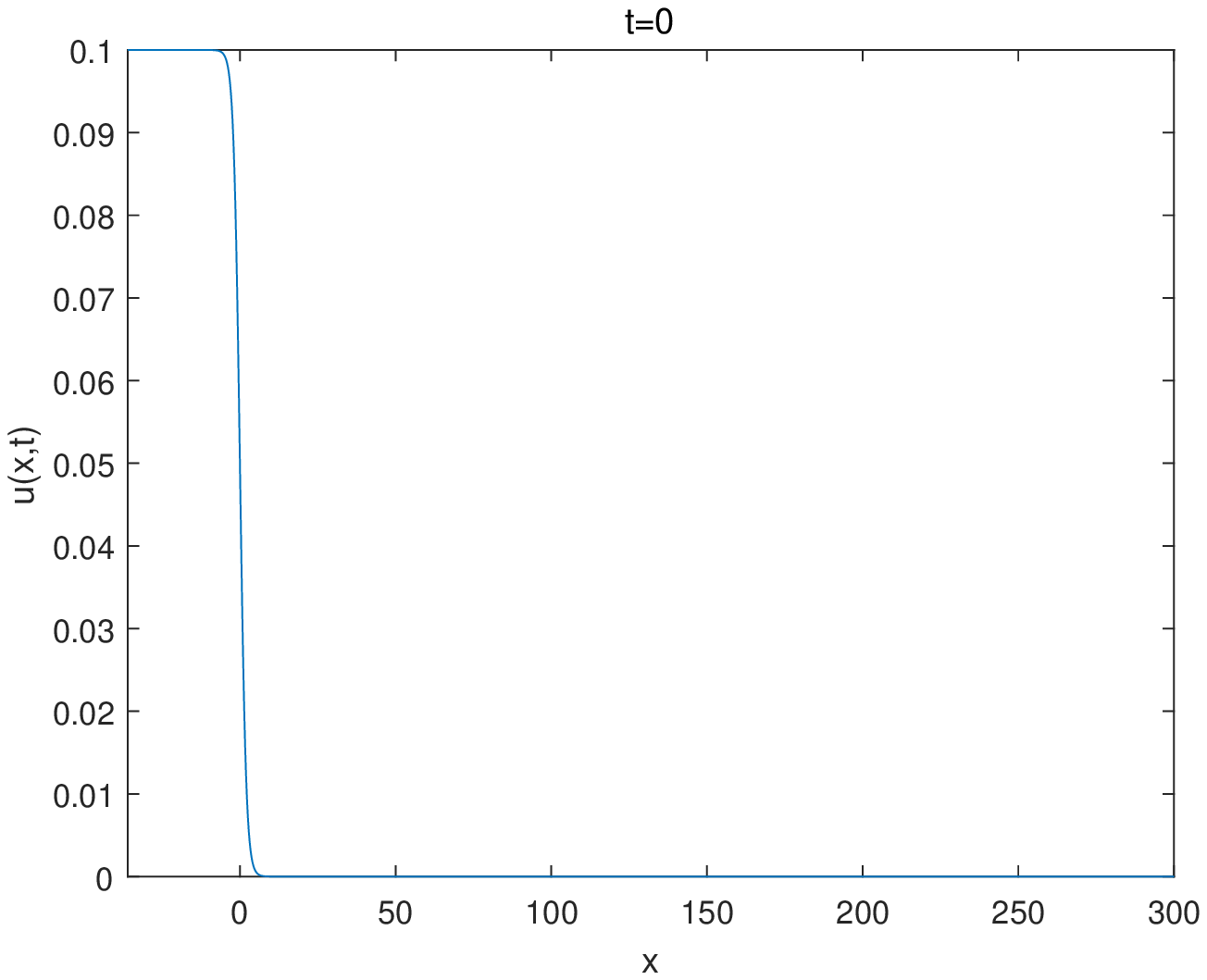}}
	\subfigure{
		\includegraphics[width=0.3\textwidth,height=0.3\textwidth]{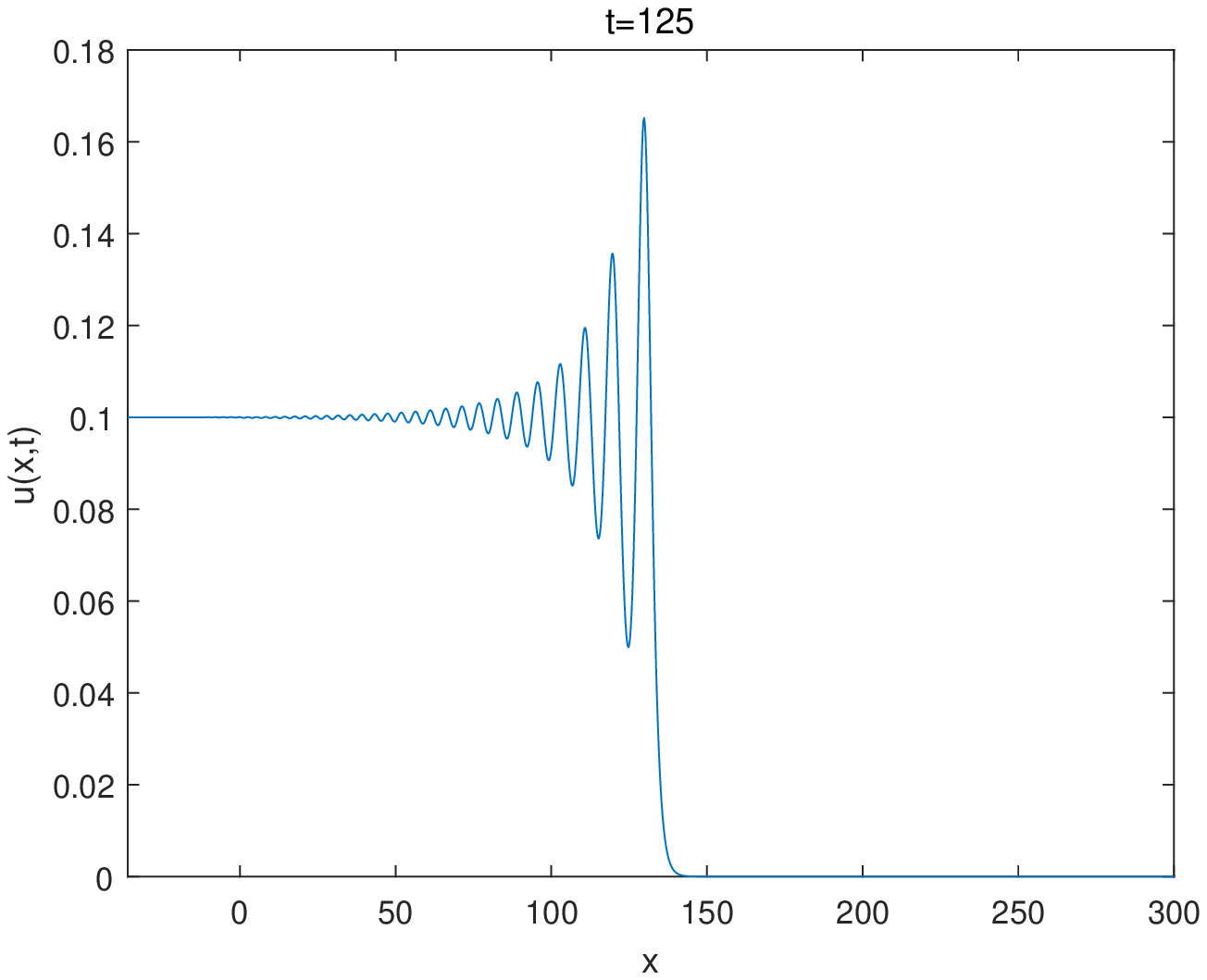}}
	\subfigure{
		\includegraphics[width=0.3\textwidth,height=0.3\textwidth]{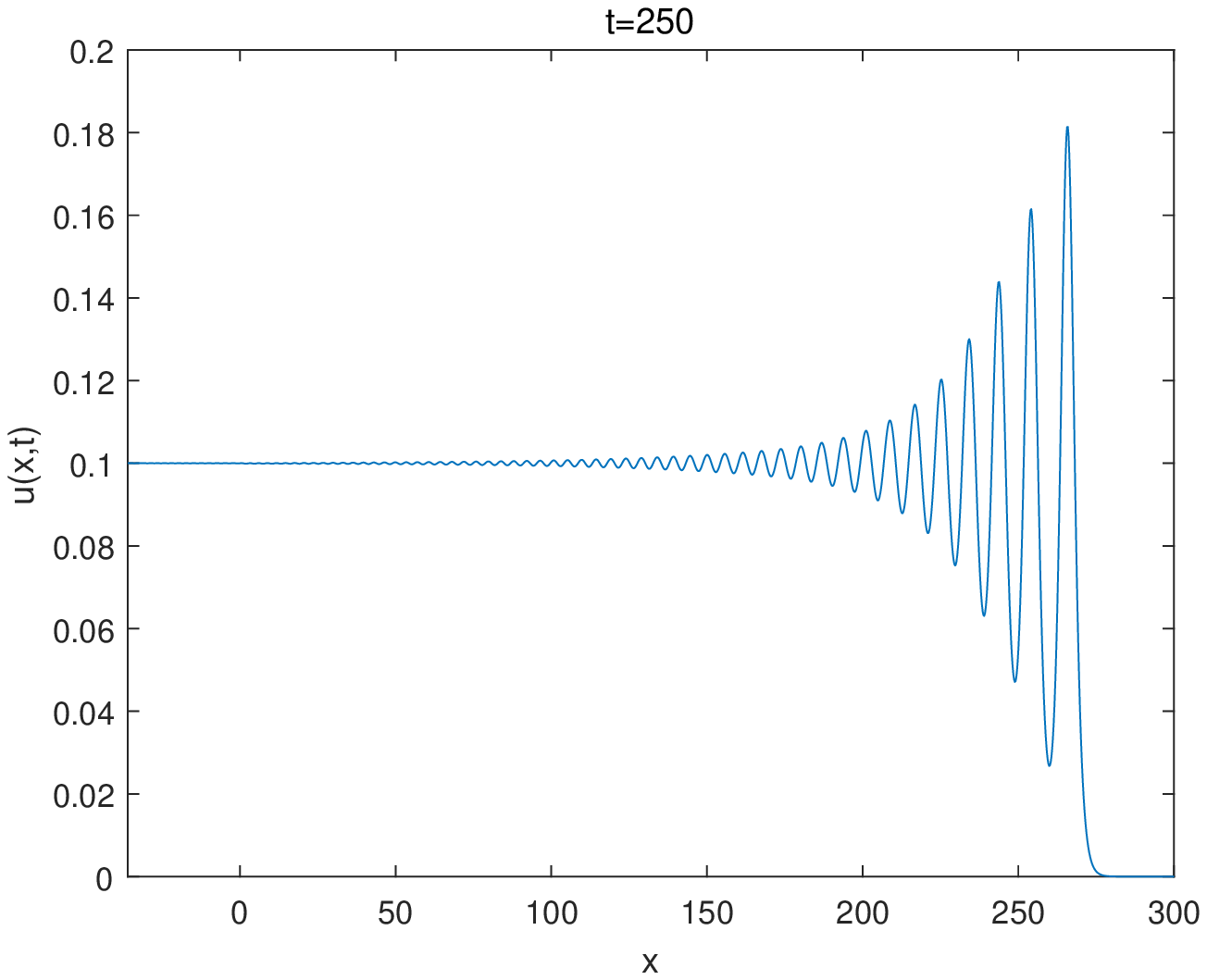}}\\
	\subfigure{
		\includegraphics[width=0.3\textwidth,height=0.3\textwidth]{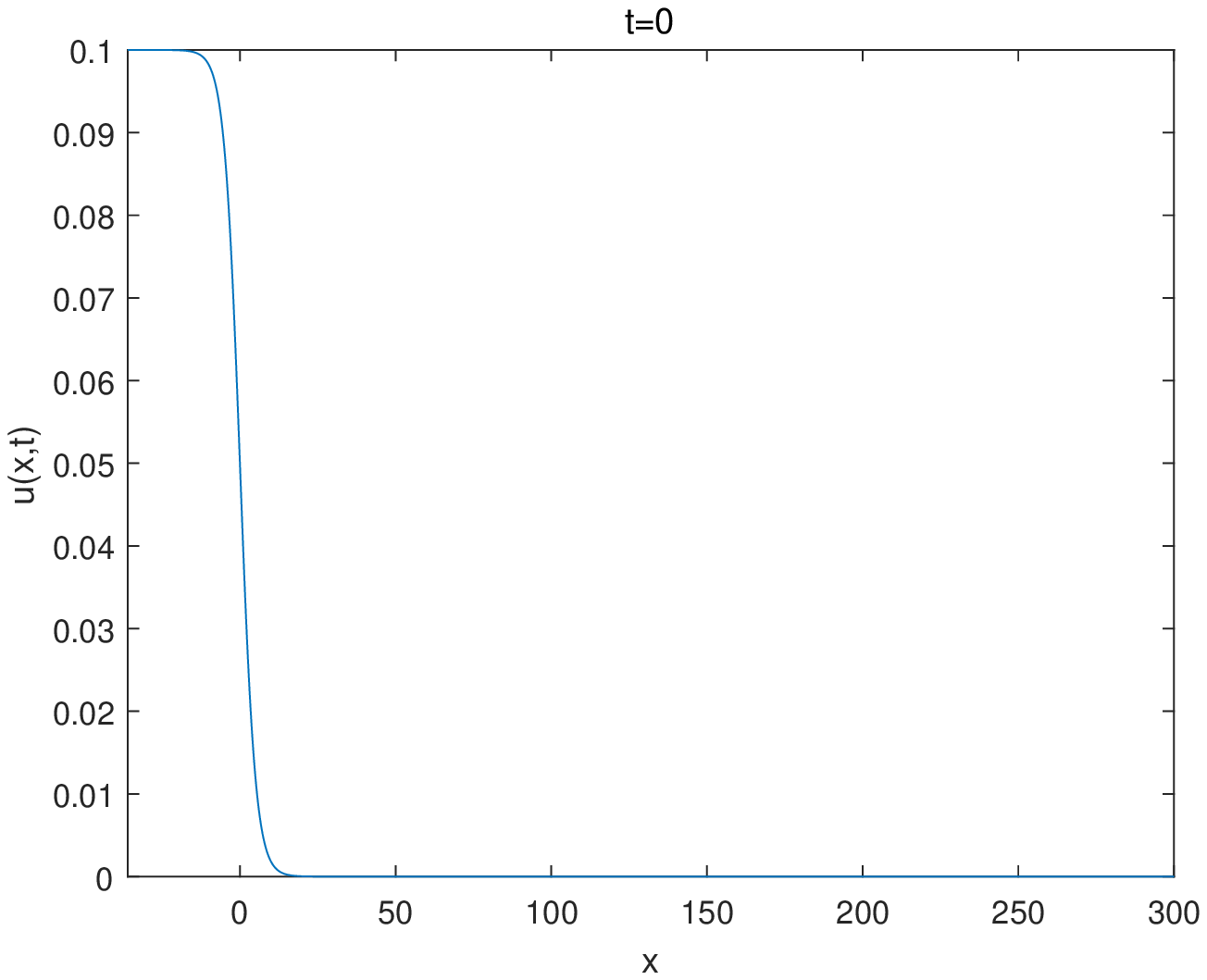}}
	\subfigure{
		\includegraphics[width=0.3\textwidth,height=0.3\textwidth]{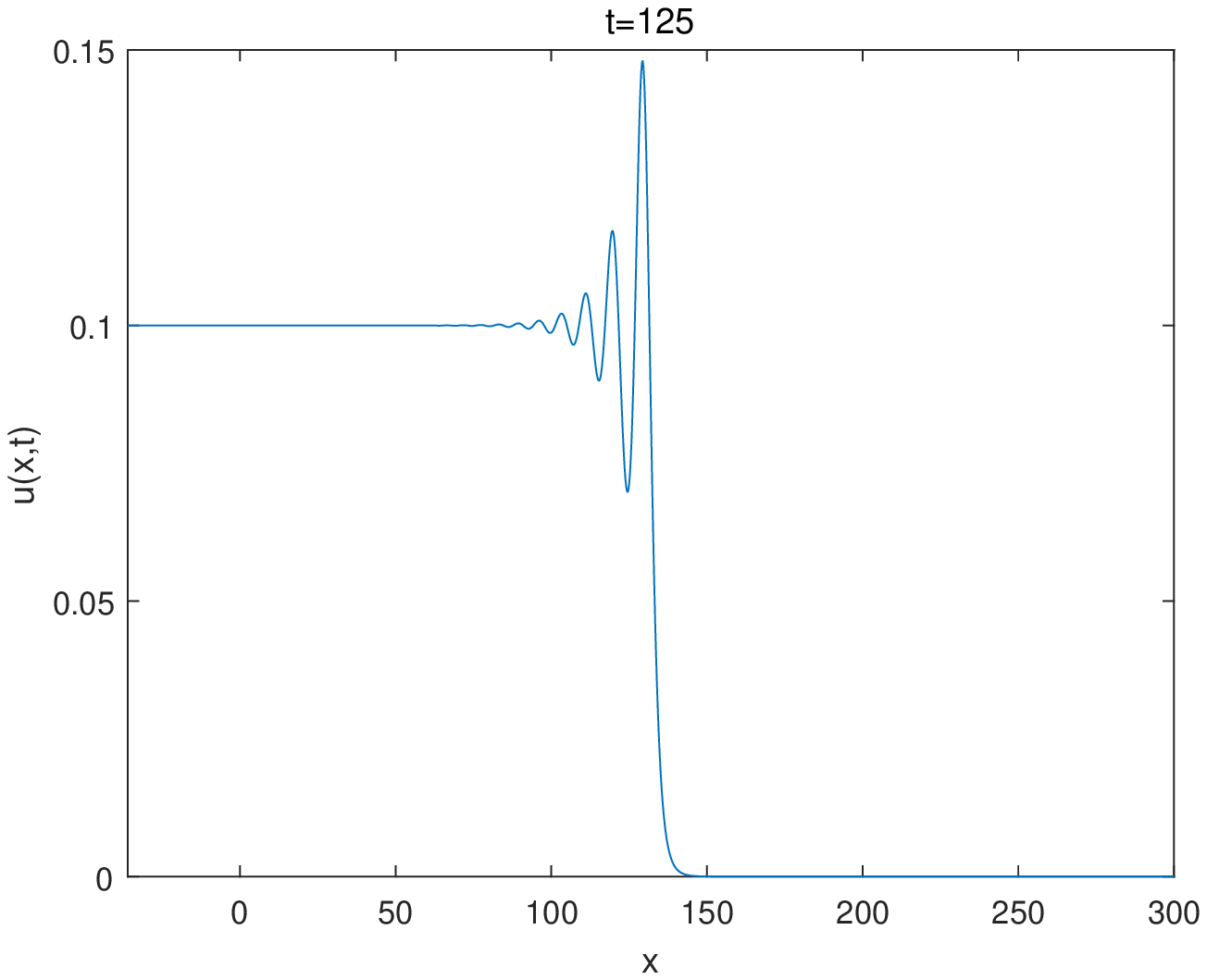}}
	\subfigure{
		\includegraphics[width=0.3\textwidth,height=0.3\textwidth]{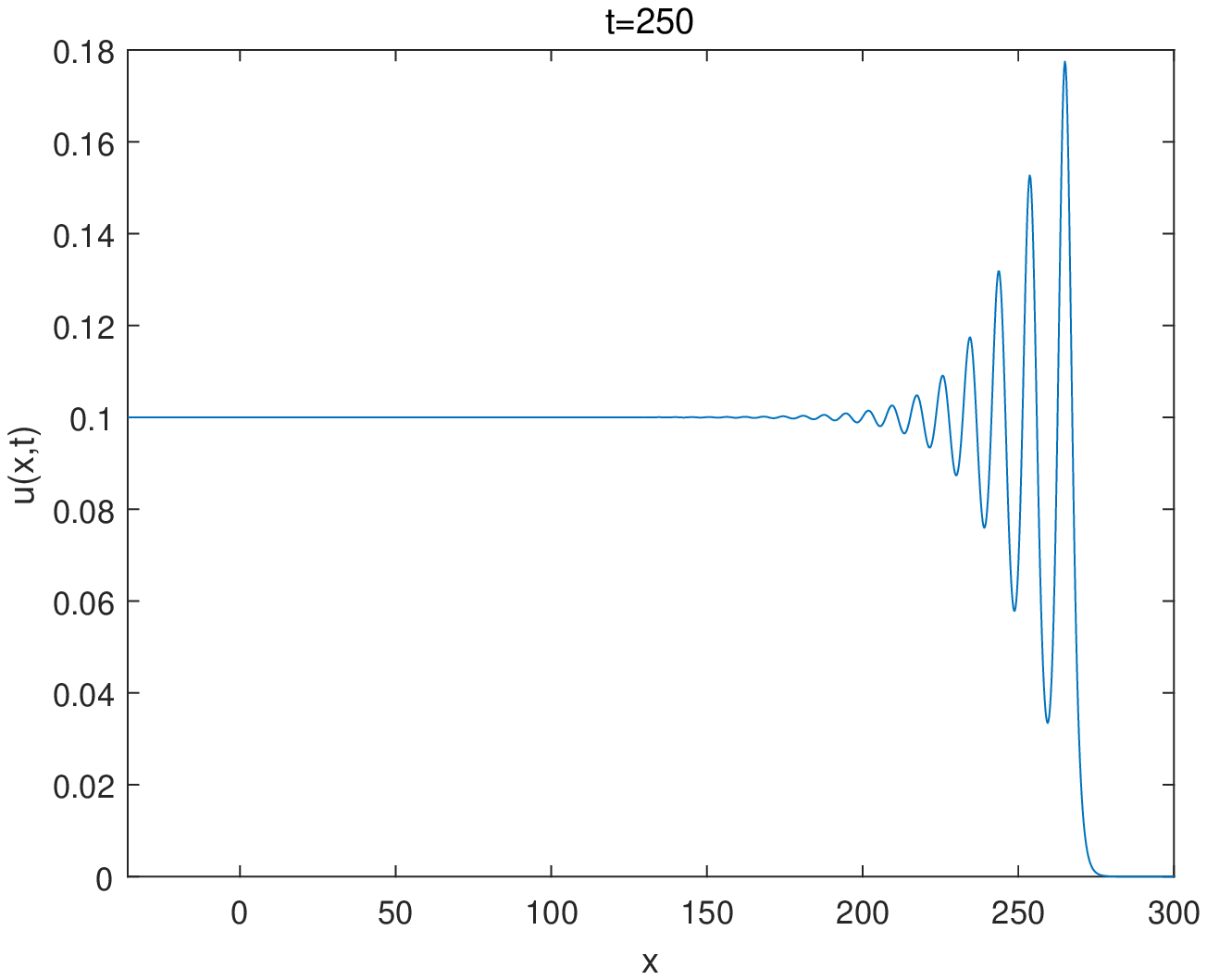}}
	\caption{\small Initial and undulation profiles with gentle $d=2$ (top) and $d=5$ (bottom) at different times using the scheme LILF.\label{fig:d2-d5}}
\end{figure}

\begin{figure}[!htbp]
	\centering
	\subfigure[]{
		\includegraphics[width=0.30\textwidth,height=0.3\textwidth]{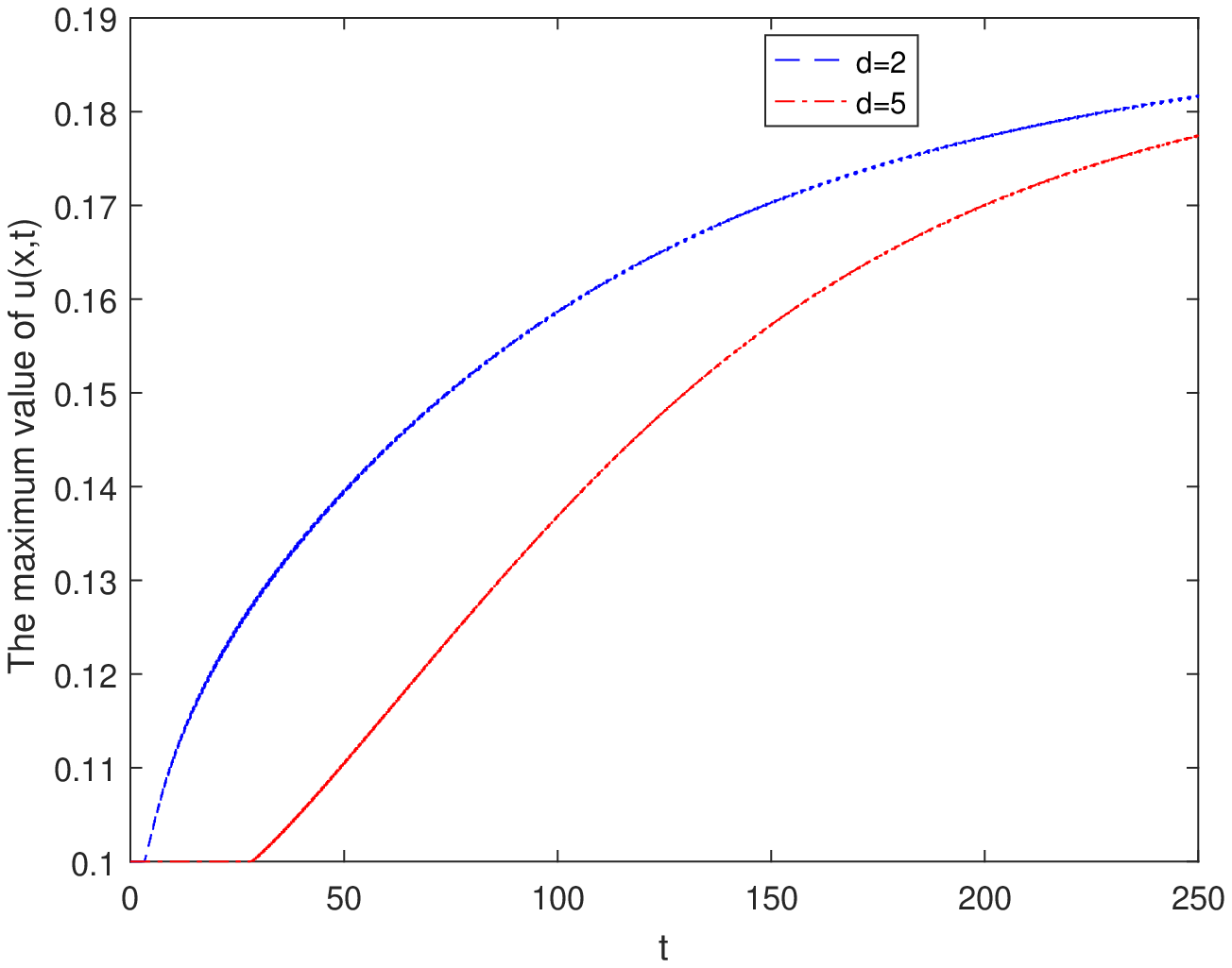}}
	\subfigure[]{
		\includegraphics[width=0.30\textwidth,height=0.3\textwidth]{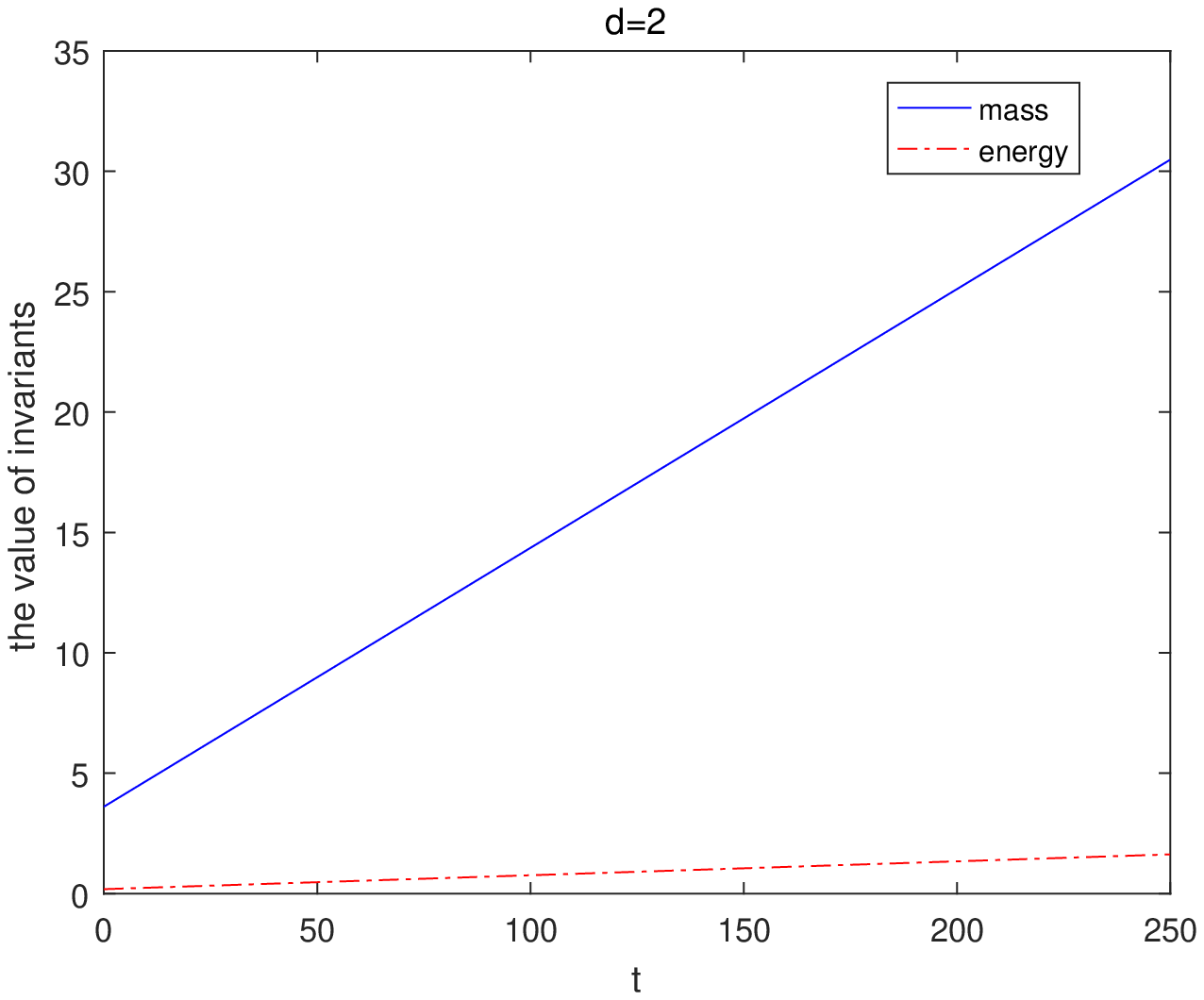}}
	\subfigure[]{
		\includegraphics[width=0.3\textwidth,height=0.3\textwidth]{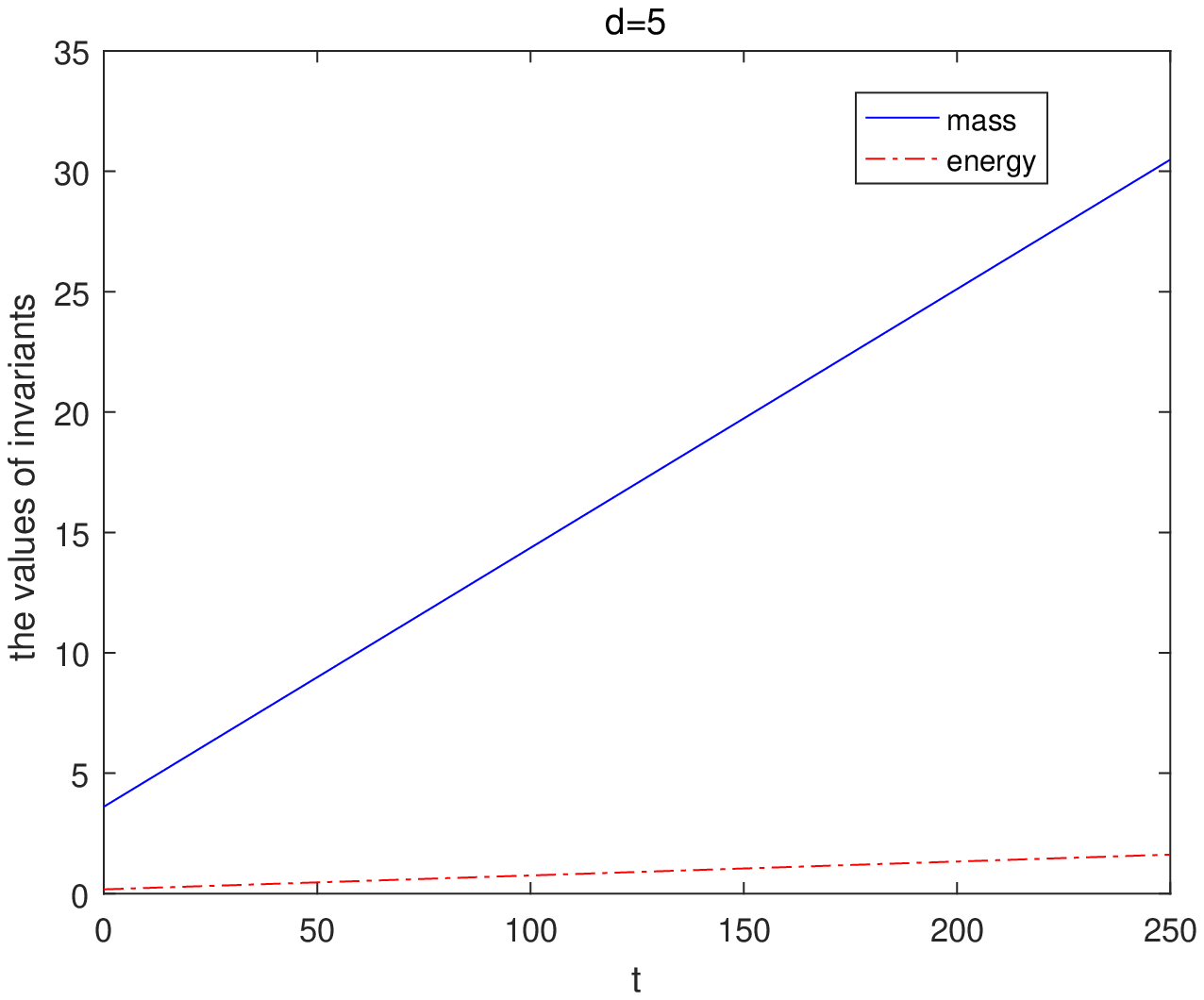}}
	\caption{\small (a) Development of the first undulation from $t=0$ to $t=250$ and (b) the behavior of the invariants for $d=2$ and (c) $d=5$ by LILF.\label{fig:mass-energy-d2-d5}}
\end{figure}

\section{Conclusions}
In this paper, incorporating the modified finite volume method with the discrete variational derivative method, linear-implicit Crank-Nicolson and linear-implicit Leap-frog scheme, we proposed and analyzed a fully implicit and three linear-implicit energy-preserving schemes for RLW equation. In this study, we consider the weak formulation and presented the modified finite volume method for the framework of the scheme. Moreover, we strictly proved that the semi-discrete system preserves a semi-discrete energy. For time discretization, the DVDM with a point and two points numerical solutions is used to preserve the semi-discrete energy obtained by the first step. The perfect combination of mFVM and DVDM results in a fully implicit scheme and a linear-implicit scheme for RLW. In addition, we develop a new idea of transforming the energy into a quadratic function
to derive energy stable schemes. Based on this strategy, we still discrete the weak form of the
equivalent formulation for the RLW equation by the so-called modified finite volume method
in space. And then we consider the linear-implicit Crank-Nicolson and  linear-implicit Leap-frog scheme  in temporal direction, two second-order
linear energy-preserving numerical schemes are obtained readily. Finally, numerical results show that our proposed numerical schemes have excellent performance
in providing accurate solution and preserving the discrete invariants. Compared with the scheme FIEP, the linear-implicit schemes LIEP, LICN and LILF improve the computational efficiency. In the future, the linear momentum-preserving scheme and the corresponding error analysis is one of our on-going project.

\end{document}